\documentclass[12pt]{amsart}

\usepackage{amsmath, amssymb,amsxtra}
\usepackage{amsthm}

\theoremstyle{plain}
\newtheorem{theorem}{\sc Theorem}[section]
\newtheorem{lemma}[theorem]{\sc Lemma}
\newtheorem{corollary}[theorem]{\sc Corollary}
\newtheorem{proposition}[theorem]{\sc Proposition}

\theoremstyle{definition}
\newtheorem{definition}[theorem]{\sc Definition}
\newtheorem{remark}[theorem]{\sc Remark}
\newtheorem{example}[theorem]{\sc Example}

\def\Xsp{X}
\def\XspN{(\Xsp, \, \|\, \cdot \, \|_\Xsp)}

\makeatletter

\title[]{On some properties of  modulation
spaces as  Banach algebras}
\author{
Hans G. Feichtinger, Masaharu Kobayashi, Enji Sato}

\address{
Hans G. Feichtinger \\
Faculty of Mathematics,
University of Vienna,
Oskar-Morgenstern-Platz 1, A-1090 Wien, Austria
and
Acoustic Research Institute,
Austrian Academy of Sciences
}
\email{hans.feichtinger@univie.ac.at}

\address{
Masaharu Kobayashi \\
Department of Mathematics,
Hokkaido University,
Kita 10, Nishi 8, Kita-Ku, Sapporo,
Hokkaido, 060-0810, Japan}
\email{m-kobayashi@math.sci.hokudai.ac.jp}

\address{Enji Sato \\
Faculty of Science,
Yamagata University,
Kojirakawa 1-4-12, Yamagata-City,
Yamagata 990-8560,
Japan}
\email{esato@sci.kj.yamagata-u.ac.jp}

\makeatother
\date{\today}

\begin{document}
\maketitle

\begin{abstract}
In this paper, we give some properties of
the modulation spaces
$M_s^{p,1}({\mathbf R}^n)$ as  commutative
Banach algebras.
In particular, we
show the Wiener-L\'evy theorem for $M^{p,1}_s({\mathbf R}^n)$,
and clarify the sets of spectral synthesis
for $M^{p,1}_s ({\mathbf R}^n)$
by using the ``ideal theory for  Segal algebras''
developed in
Reiter \cite{Reiter}.
The inclusion relationship between
the modulation space $M^{p,1}_0 ({\mathbf R})$ and the Fourier
Segal algebra ${\mathcal F}\hspace{-0.08cm}A_p({\mathbf R})$ is also  determined.
\end{abstract}


\section{Introduction}

A commutative Banach algebra is a Banach space $\XspN$
which is  a commutative algebra, i.e., with an
 associative, distributive and commutative multiplication
satisfying $\| fg \|_\Xsp \leq \| f \|_\Xsp  \| g \|_\Xsp$
for all $f,g \in \Xsp$.
There are many examples of commutative  Banach algebras.
For this paper 
the Fourier algebra  $A({\mathbf R})$ is an important
example, i.e., the set of  all functions on ${\mathbf R}$
which are the Fourier transforms of functions
in $L^1({\mathbf R})$,  is a commutative
Banach algebra with pointwise multiplication, denoted
by $\mathcal{F}L^1(\mathbf R)$ elsewhere.
In addition, the Wiener algebra $A({\mathbf T})$,
which corresponds to a periodic version
of $A({\mathbf R})$, and
the Fourier-Beurling algebra
${\mathcal F}L_s^1({\mathbf R}^n)$ $(s \geq 0)$,
which can be regarded as a  generalization of $A({\mathbf R})$,
are  commutative Banach algebras
(see Section \ref{subsection Fourier-Beurling} for the definition of
${\mathcal F}L^1_s ({\mathbf R}^n)$).
As is well known, those function spaces play important roles
in various fields such as commutative Banach algebras,
operating functions, and  spectral synthesis.
We refer the reader to Kahane \cite{Kahane},
Katznelson \cite{Katznelson}, Reiter \cite{Reiter},
Reiter-Stegeman \cite{Reiter-Stegeman 2000}
 and Rudin \cite{Rudin}
for more details.

On the other hand, modulation spaces $M^{p,q}_s({\mathbf R}^n)$
form a family of function spaces
introduced by Feichtinger \cite{Feichtinger}
(see Definition \ref{def of modulation spaces}).
In some  sense, they behave like the Besov spaces $B^{p,q}_s({\mathbf R}^n)$.
But they appear to be better suited for the  description of
 problems in the area of time-frequency analysis
and are often a good substitute for the usual
spaces  $L^p({\mathbf R}^n)$ or $B^{p,q}_s({\mathbf R}^n)$
(see \cite{Feichtinger-Strohmer}, \cite{Grochenig 2001}
for more details).
It is important for our considerations  that  $M^{p,q}_s({\mathbf R}^n)$
is a  commutative Banach algebra
if  $s > n( 1 - \frac{1}{q} )$, or $q=1$ and $s \geq 0$
(cf.\ \cite{fe90}, Theorem 10),
and that it appears naturally in the study
of certain partial differential equations.
For example,  we can solve the nonlinear Schr\"odinger equations
$
i u_t + \Delta u =  | u |^{2k} u
$
in $M^{2,1}_s ({\mathbf R}^n)$ for all
$k \in {\mathbf N}$ by using
the  algebraic properties of  $M^{p,q}_s ({\mathbf R}^n)$
(see \cite{Benyi-Okoudjou}, \cite{Wang Huo Hao Guo}).
Furthermore, as a commutative Banach algebra,
$M^{p,q}_s ({\mathbf R}^n)$
satisfies many  interesting properties  similar to
$A({\mathbf R})$, $A({\mathbf T})$ and also
${\mathcal F}L^1_s({\mathbf R}^n)$
(cf.  \cite{Bhimani}, \cite{Bhimani Ratnakumar},
\cite{Kobayashi Sato 1},
\cite{Kobayashi Sato 2}, \cite{Kobayashi Sato 3},
\cite{Okoudjou}).
Therefore, it is of interest to further clarify the
properties of $M_s^{p,1}({\mathbf R}^n)$.

In this paper  we study several properties of $M_s^{p,1}({\mathbf R}^n)$ as a commutative  Banach algebra. In particular, we will establish the validity of
the Wiener-L\'evy theorem  for $M^{p,1}_s({\mathbf R}^n)$,
and see that the sets of spectral synthesis for
${\mathcal F}L^1_s ({\mathbf R}^n)$
coincide with those for $M^{p,1}_s ({\mathbf R}^n)$,
independently from $p$ (given the general restrictions).
Moreover, we will consider the  inclusion relation between
the modulation space $M^{p,1}_0 ({\mathbf R})$ and the
Fourier Segal algebra ${\mathcal F}\hspace{-0.08cm}A_p({\mathbf R})$.

 The organization of this paper is as follows. After a preliminary section devoted to the definition and basic properties of
$M^{p,q}_s({\mathbf R}^n)$ and ${\mathcal F}L^1_s ({\mathbf R}^n)$
we will demonstrate in Section \ref{section Approximate units}
that there exist approximate units for $M_s^{p,q}({\mathbf R}^n)$
(see Theorem \ref{f-psif}).
In Section \ref{closed ideals in Mp1s}
we consider the closed ideals in
 $M^{p,1}_s({\mathbf R}^n)$ and prove that closed modulation
 invariant subspaces of $M^{p,1}_s({\mathbf R}^n)$
 coincide with closed ideals in  $M^{p,1}_s({\mathbf R}^n)$.
 Section \ref{Wiener-Levy type theorems}
is devoted to the Wiener-L\'evy theorem,
which was originally proved
by Wiener  \cite{Wiener}  and by  L\'evy \cite{Levy} for $A({\mathbf T})$.
In particular we will  show
that the Wiener-L\'evy theorem and also
Wiener's general Tauberian theorem hold for $M^{p,1}_s({\mathbf R}^n)$
if  $1 \leq p < \infty$ and $s \geq 0$
(see Theorems \ref{image of analytic function}
and \ref{ideal generated by f}).
In Section \ref{Set of spectral synthesis in Mp1s}
we will consider the set of spectral synthesis
for $M^{p,1}_s({\mathbf R}^n)$.
By using the ``ideal theory for Segal algebras'' developed in
Reiter \cite{Reiter}  we will clarify the   relation between
the sets of spectral synthesis for $M^{p,1}_s({\mathbf R}^n)$
and the sets of spectral synthesis for
${\mathcal F}L^1_s({\mathbf R}^n)$ (see Theorem
\ref{equality of the set of spectral synthesis}).
In Section \ref{Appendix} we will prove that
single points of ${\mathbf R}^n$ are sets of spectral synthesis
for $M^{p,1}_s({\mathbf R}^n)$
without using Theorem \ref{equality of the set of spectral synthesis}.
Finally, in Section \ref{A relation between Mp10}
we will determine the inclusion relation between $M_0^{p,1}({\bf R}^n)$ and the
Fourier Segal algebras ${\mathcal F}\hspace{-0.08cm}A_p({\mathbf R}^n)$.

\section{Preliminaries}

The following notation will be used throughout this article. We use $C$ to denote various positive constants which may change from line to line.
We use the notation  $I \lesssim J$ if $I$ is bounded by a constant times $J$
and we denote $I \approx J$ if $I \lesssim J$ and $J \lesssim I$.
The closed ball with center $x_0 \in {\mathbf R}^n$
and radius $r>0$ is defined by
$B_r (x_0) = \{ x \in {\mathbf R}^n ~|~ |x-x_0| \leq r \}$.
For $x \in {\mathbf R}^n$, we write $\langle x \rangle = (1+|x|^2)^{\frac{1}{2}}$.
We define for $1 \leq p < \infty$
and $s \in {\mathbf R}$
$$
\| f \|_{L^p_s} =
\Big(
\int_{{\mathbf R}^n}
\big(  \langle x \rangle^s  |f(x)|  \big)^p dx
\Big)^{\frac{1}{p}},
$$
and
$ \| f \|_{L^\infty_s} =
{\rm ess.sup}_{x \in {\mathbf R}^n} \langle x \rangle^s   |f(x)|$.
We simply write $L^p ({\mathbf R}^n)$
instead of $L^p_0({\mathbf R}^n)$.
For $1 \leq p < \infty$, we denote by
$p^\prime$ the conjugate exponent of $p$,
i.e., $\frac{1}{p} + \frac{1}{p^\prime} =1$.
We write $C_c^\infty({\mathbf R}^n)$ to denote the set
of all complex-valued infinitely differentiable functions on ${\mathbf R}^n$ with compact support.
We write ${\mathcal S}({\mathbf R}^n)$ to denote the Schwartz space of
all complex-valued rapidly decreasing infinitely differentiable functions on ${\mathbf R}^n$
and ${\mathcal S}^\prime({\mathbf R}^n)$ to denote the space of tempered
distributions on ${\mathbf R}^n$.
We use $\langle F, G \rangle$ to denote the extension
of the inner product
$\langle F,G \rangle =
 \int  F(t) \overline{G(t)} dt$ on $L^2$
to
${\mathcal S}^\prime \times {\mathcal S}$
or $M^{p,q}_s \times M^{p^\prime, q^\prime}_{-s} $
(see Lemma \ref{basic pro mod} (ii) below).
The {\it Fourier transform} of $f \in L^1({\mathbf R}^n) $ is defined by
$$
{\mathcal F} f (\xi) = f^\wedge(\xi) = \widehat{f} (\xi)  = \int_{{\mathbf R}^n}  f(x) e^{-ix \xi} dx.
$$
Similarly,  if $h \in L^1({\mathbf R}^n) $ then the inverse Fourier transform of $h$ is defined by  
${\mathcal F}^{-1} h(x) 
= (2 \pi)^{-n}\widehat{h} (-x)$.
We note that
$\langle f, g \rangle_{ L^2({\mathbf R}^n) }
= (2 \pi)^{-n} \langle \widehat{f}, \widehat{g} \rangle_{L^2({\mathbf R}^n)} $,
$(f*g)^\wedge = \widehat{f}  \widehat{g}$ and
$(fg)^\wedge = (2 \pi)^{-n} ( \widehat{f} * \widehat{g}  )$,
where
$(f*g) (x) = \int_{{\mathbf R}^n} f(x-y) g(y) dy $.
Moreover,
$M({\mathbf R}^n)$ denotes  the sets of
all  bounded regular Borel measures $\mu$ on ${\mathbf R}^n$ with
the norm
$$
\| \mu  \|_{ M({\mathbf R}^n)}
= \sup_{f \in C_c ({\mathbf R}^n), \| f \|_{L^\infty} \leq 1  }
\Big|  \int_{{\mathbf R}^n} f(x) d \mu(x) \Big|.
$$
The Fourier-Stieltjes transform of $\mu$ is defined by
$\widehat{\mu} (\xi) = \int_{{\mathbf R}^n} e^{-ix \xi} d \mu(x)$.
Here $C_c({\mathbf R}^n)$ denotes the set of all $f \in C({\mathbf R}^n)$
with compact support.
Let  $\delta_a$ $(a \in {\mathbf R}^n)$ be the unit mass concentrated at the
point $x=a$, i.e.,  $\delta_a(E) =1$ if $a \in E$ and $\delta_a(E)=0$ otherwise.
Then we have
$\delta_a \in M({\mathbf R}^n)$,
 $\| \delta_a \|_{M({\mathbf R}^n)}=1$, $\widehat{\delta_a} (\xi) = e^{-ia \xi}$
and
$\delta_a * \delta_b = \delta_{a+b}$  $(a,b \in {\mathbf R}^n)$
(see \cite[Ch.1.3]{Rudin} for more details).
For two Banach spaces $B_1$ and $B_2$,
$B_1 \hookrightarrow B_2$ means that
$B_1$ is continuously embedded into $B_2$.

\subsection{Short-time Fourier transform}
\label{STFT}
For $f \in {\mathcal S}^\prime ({\mathbf R}^n)$
and $\phi \in {\mathcal S} ({\mathbf R}^n)$.
the short-time Fourier transform
$V_\phi$ of $f$ with respect to the window
$\phi$ is defined by the duality expression
$$
V_\phi f (x, \xi)
=
\langle f, M_\xi T_x \phi \rangle
=
\langle f(t), \phi(t-x) e^{it \xi}  \rangle
=
\int_{{\mathbf R}^n}
f(t)  \overline{\phi(t-x)} e^{-i t \xi} dt,
$$
where the translation operator $T_x$ and the modulation operator $M_\xi$
are defined by $(T_x h)(t) = h(t-x)$ and $(M_\xi h)(t)= e^{it \xi} h(t)$, respectively.

It is known that
$V_\phi f  \in C( {\mathbf R}^{n} \times {\mathbf R}^n  )$
(see \cite[Lemma 11.2.3]{Grochenig 2001})
and
\begin{align}
\label{basic STFT}
V_\phi f (x, \xi)
&=
(2 \pi)^{-n}  e^{- i x \xi}
V_{ \widehat{\phi}  } \widehat{f} (\xi, -x)
=
(2 \pi)^{-n} e^{-i x \xi} (f* M_\xi \phi^*) (x),
\end{align}
where $\phi^* (x) = \overline{\phi(-x)}$
(see \cite[Lemma 3.1.1]{Grochenig 2001}).
We also note that
if $f \in {\mathcal S}  ({\mathbf R}^n)$,
then $V_\phi f \in {\mathcal S}({\mathbf R}^n \times {\mathbf R}^n)$
(see \cite[Theorem  11.2.5]{Grochenig 2001}).
Moreover, we have Moyals Equation:
$$
\langle V_\phi f, V_\psi g \rangle_{L^2({\mathbf R}^{2n})}
=(2 \pi)^n \langle \psi, \phi \rangle_{L^2({\mathbf R}^n)}
\langle f, g \rangle_{L^2({\mathbf R}^n)}
$$
for all $f,g , \phi, \psi \in L^2({\mathbf R}^n)$
(see \cite[Theorem  3.2.1]{Grochenig 2001}).

\subsection{Modulation spaces}
\label{sec def of modulation}

\begin{definition}[{\cite{Feichtinger}}]
\label{def of modulation spaces}
Let $1 \leq  p,q \leq \infty$, $s \in {\mathbf R}$ and $\phi \in {\mathcal S} ({\mathbf R}^n) \setminus  \{ 0 \}$.
The modulation space $M_s^{p,q}({\mathbf R}^n) = M_s^{p,q}$
consists of all $ f \in {\mathcal S}^\prime({\mathbf R}^n)$ such that the norm
$$
\| f \|_{M_s^{p,q}} = \Big(
\int_{{\mathbf R}^n} \langle \xi  \rangle^{sq}
\Big(
\int_{{\mathbf R}^n} | V_\phi f ( x, \xi) |^p dx \Big)^{\frac{q}{p}}
d \xi \Big)^{\frac{1}{q}}
$$
is finite (with usual modification if $p=\infty$ or $q=\infty$).
We simply write $M^{p,q} ({\mathbf R}^n) $
instead of  $M^{p,q}_0 ({\mathbf R}^n) $.
\end{definition}

We collect basic properties of $M^{p,q}_s ({\mathbf R}^n)$
in the following lemma:
\begin{lemma}
\label{basic pro mod}
Let $1 \leq p,p_1,p_2, q,q_1,q_2 \leq \infty$ and $s,s_1,s_2 \in {\mathbf R}$.
Then
\begin{itemize}

\item[$(i)$] The space $M^{p,q}_s ({\mathbf R}^n)$
is a Banach space, and different windows define equivalent norms.

\item[$(ii)$] $($Density and duality$)$
If  $p,q< \infty$, then ${\mathcal S} ({\mathbf R}^n)$
is dense in $M^{p,q}_s ({\mathbf R}^n)$ and
$(M^{p,q}_s ({\mathbf R}^n) )^\prime = M^{p^\prime,q^\prime}_{-s} ({\mathbf R}^n)$.

\item[$(iii)$]
If  $p_1 \leq p_2$, $q_1 \leq q_2$ and
$s_1 \geq s_2$, then $M^{p_1,q_1}_{s_1} ({\mathbf R}^n) \hookrightarrow
M^{p_2,q_2}_{s_2} ({\mathbf R}^n)$.

\item[$(iv)$]
If $s> \frac{n}{q^\prime}$, or $q=1$ and $s \geq 0$, then
$M^{p,q}_s ({\mathbf R}^n)  \subset C ({\mathbf R}^n)$
and
 $M^{p,q}_s ({\mathbf R}^n)$
are multiplication algebra, i.e., we have
\begin{align*}
\label{multiplication const}
\| fg \|_{M^{p,q}_s}
\leq c \| f \|_{M^{p,q}_s} \| g \|_{M^{p,q}_s},
\quad
f,g \in M^{p,q}_s  ({\mathbf R}^n)
\end{align*}
for some $c \geq 1$.

\item[$(v)$]
If $s \geq 0$, then  there exists $C >0$ such that
$$
\| f_\lambda \|_{M^{\infty,1}_s}
\leq C \| f \|_{M^{\infty,1}_s} \quad \forall
f \in M^{\infty,1}_s ({\mathbf R}^n)
$$
for all  $0< \lambda \leq 1$, where $f_\lambda (x) = f(\lambda x)$.
\end{itemize}
\end{lemma}
$(i)$-$(v)$ are announced in
\cite[Theorem 3.2]{Cordero Okoudjou},
\cite[Section 6]{Feichtinger},
\cite{Grobner},
\cite{Toft 2004},
\cite{Toft 2004-2}
and summaries are given in \cite{Benyi-Okoudjou},
\cite{Cordero Rodino}, \cite[Chapter 11]{Grochenig 2001}.

We also note that there is another characterization of
the modulation spaces using BUPUs on the Fourier
transform side (see \cite{Feichtinger 83},\cite{fe90} or
\cite{Wang Huo Hao Guo}):
Assume that $\varphi \in {\mathcal S} ({\mathbf R}^n)$
satisfies
$$
{\rm supp~} \varphi \subset [-1,1]^n
\quad
{\rm and} \quad
\sum_{k \in {\mathbf Z}^n}
\varphi (\xi -k) =1 \quad
{\rm for~all ~}
\xi \in {\mathbf R}^n.
$$
Then we have, writing
$\varphi(D-k)f  =
{\mathcal F}^{-1} (T_k \varphi \cdot \widehat{f} )  $:
$$
\| f \|_{M^{p,q}_s}
\approx
\Big(
\sum_{k \in {\mathbf Z}^n}
\langle k \rangle^{sq}
\Big(
\int_{{\mathbf R}^n}
| \varphi(D-k) f (x) |^p dx
\Big)^{ \frac{q}{p} }
\Big)^{\frac{1}{q}}
$$
with obvious modifications if $p$ or $q= \infty$.

\begin{remark}
\label{equivalent norm of Mpqs}
We note that 
we may assume that
$0 \leq \varphi(\xi) \leq 1$ $(\xi \in {\mathbf R}^n)$
and  $\varphi (\xi)  =1$
on the box $[- \frac{1}{10} , \frac{1}{10}   ]^n$.
\end{remark}

\begin{remark}
Let $1 \leq q \leq \infty$, $s \in {\mathbf  R}$ and
$(B, \| \cdot \|_B )$ be a Banach space of tempered distributions
on ${\mathbf  R}^n$ such that $ {\mathcal S} \cdot B \subset B$.
Using BUPUs the {\it Wiener amalgam space}
$W( B,\ell^q_s)({\mathbf  R}^n )$ is defined by the norm
$$
\| f \|_{W (B,\ell^q_s)}
=
\Big(
\sum_{k \in {\mathbf Z}^n}
\langle k \rangle^{sq}
\|  f \cdot T_k  \varphi \|_{B}^q
\Big)^{\frac{1}{q}} < \infty.
$$
We note that
 $W ({\mathcal F} L^p,\ell^q_s)({\mathbf  R}^n )
= {\mathcal F}^{-1} (M^{p,q}_s ({\mathbf R}^n))$,
where ${\mathcal F} L^p ({\mathbf R}^n) (= {\mathcal F} L^p)$ denotes the
Fourier-Lebesgue space with the norm
$\| f \|_{{\mathcal F}L^p} = \| \widehat{f} \|_{L^p}$
(see also Section \ref{subsection Fourier-Beurling} below).
\end{remark}

The following lemma seems to be known to many people,
but for the reader's convenience, we  give the
proof (cf. \cite{fe90}, {\cite{Guo Fan Dashan Wu Zhao}} or \cite{Teofanov Toft}).
\begin{lemma}
\label{product estimate}
Let $1 \leq p,q  < \infty $.
Suppose that
 $s > {n}/{q^\prime}$,
or  $q=1$ and $s \geq 0$.
Then we have
$$
\| fg \|_{M^{p,q}_s}
\lesssim \| f \|_{M^{\infty,q}_s}
\| g \|_{M^{p,q}_s}, \quad f,g \in {\mathcal S}({\mathbf R}^n).
$$
If
$f \in M^{p,1}_s ({\mathbf R}^n)$ and $g \in M^{p^\prime, \infty}_{-s} ({\mathbf R}^n)$,
then $f g \in
M^{1, \infty}_{-s} ({\mathbf R}^n)
(\hookrightarrow
M^{p^\prime, \infty}_{-s} ({\mathbf R}^n) )$.
\end{lemma}
\begin{proof}
We start the proof by observing that the pointwise relationships
for modulation spaces correspond to convolution relations for
their (inverse) Fourier transforms
in $W ({\mathcal F} L^p,\ell^q_s)({\mathbf  R}^n )$.

The claim made is thus equivalent to the statement that (writing
$f,g$ for $\widehat{f}$, $\widehat{g}$ respectively):
\begin{equation}\label{WAMSconv}
   \| f \ast g \|_{W ({\mathcal F} L^p,\ell^q_s)} \lesssim
    \| f \|_{W ({\mathcal F} L^\infty,\ell^q_s)}   \| g \|_{W ({\mathcal F} L^p,\ell^q_s)}
\end{equation}
which results by the coordinate-wise convolution result established in
\cite{Feichtinger 83-2} and the fact that
$$   {\mathcal F}L^p \ast  {\mathcal F} L^\infty
\subset {\mathcal F} L^p \quad\mbox{because} \quad
      L^p \cdot L^\infty \subseteq L^p,      $$
as well as the fact that $(\ell^q_s ({\mathbf Z}^n) , \| \cdot \|_{\ell^q_s })$
is a Banach algebra with respect to the natural convolution, because the weight
function $x \mapsto \langle  x  \rangle^s$ is weakly subadditive for any
$s \geq 0$ and consequently $\ell^1 ({\mathbf Z}^n) \cap \ell^q_s ({\mathbf Z}^n)$ is a Banach
algebra with respect to convolution. This observation have been
made explicit in \cite{Brandenburg}. But for $s > \frac{n}{q^\prime}$, it follows
from the H\"older inequality that
$\ell^q_s  ({\mathbf Z}^n) \hookrightarrow  \ell^1 ({\mathbf Z}^n)$ (with corresponding
continuous embedding), and thus
 $(\ell^q_s ({\mathbf Z}^n) , \| \cdot \|_{\ell^q_s })$ is Banach algebra
with respect to (discrete) convolution.

In a similar way we use the local convolution relation
$$
  {\mathcal F}L^p \ast {\mathcal F}L^{p^\prime} \subset  {\mathcal F}L^\infty
\quad \mbox{due to the relation}
  \quad    L^p \cdot L^{p^\prime}  \subseteq L^\infty,
$$
combined with the global condition
$$
 \ell^p_{s} \ast \ell^{\, p'}_{-s} \subseteq \ell^{\infty}_{-s}
$$
using the submultiplicativity condition (writing $x = (x-y)+y$):
$$  \langle {x+y}  \rangle^{-s}  \leq \langle  x \rangle^{-s}  \langle y \rangle^s. $$
\end{proof}

Let us also observe another consequence of the reasoning used in the
last proof, which allows to provide some insight into the pointwise
multiplier algebra of Sobolev algebras. We formulate it as another
lemma. Let us shortly recall that the classical Sobolev spaces
$(H^s({\mathbf R}^n), \| \cdot \|_{H^s} )$ are just inverse images of (polynomially) weighted
$L^2$-spaces,
and thus appear in the family of modulation spaces as
$$ H^s =  {\mathcal F}^{-1} (L^2_s )
=  {\mathcal F}^{-1} (W(  {\mathcal F}^{-1} (L^2), \ell^2_s  ))
= {\mathcal F}^{-1} (W(   L^2, \ell^2_s  ))
 $$
so that one has $\ell^2_s ({\mathbf Z}^n) \hookrightarrow \ell^1 ({\mathbf Z}^n) $ for $s > \frac{n}{2}$ via
the Cauchy-Schwarz inequality and consequently by the Hausdorff-Young
version for Wiener amalgams (see \cite{fe81-1})  one has the following chain
of continuous embeddings:
 $$H^s  \hookrightarrow   {\mathcal F}^{-1}
(W(L^2, \ell^1)) \hookrightarrow W( {\mathcal F} L^1, \ell^2 )
\hookrightarrow W(C_0, \ell^2).
$$
\begin{lemma} \label{SobPMult}
For $s > {n}/{2}$ we have  $M^{\infty,2}_s \cdot H^s \subseteq H^s$, meaning
that $M^{\infty,2}_s({\mathbf R}^n)$ is a subset of  the pointwise multiplier algebra of the Sobolev algebra $H^s ({\mathbf R}^n)$.
\end{lemma}
\begin{proof}
We have already noted that $s > {n}/{2}$  implies  $H^s  ({\mathbf R}^n)
\hookrightarrow W(C_0, \ell^2) ({\mathbf R}^n)$.
The following reasoning implies the classical fact that it is a
Banach algebra with respect to pointwise multiplication (denoted
as Sobolev algebra).

We have to verify that the corresponding convolution relation
is valid on the Fourier transform side.  In fact, we may invoke
the main result of \cite{Feichtinger 83-2}:
\begin{equation}\label{Wtconvrel1}
W ({\mathcal F} L^\infty, \ell^2_s) * W( {\mathcal F} L^2, \ell^2_s)
\subset W(L^2, \ell^2_s),
\end{equation}
using again that $\ell^2_s ({\mathbf Z}^n) = \ell^2_s ({\mathbf Z}^n) \cap \ell^1({\mathbf Z}^n) $
is a Banach algebra with respect to convolution, and
$$   {\mathcal F} L^\infty *  {\mathcal F} L^2
\subseteq  {\mathcal F} L^2 = L^2,
\quad \mbox{due to the fact that} \quad L^\infty \cdot L^2  \subseteq L^2.$$
\end{proof}

In order to make our paper self-contained
let us recall the following
result which makes use of elementary inclusion results for Wiener
amalgams and the
Hausdorff-Young theorem for generalized amalgam spaces as
described in \cite[Theorem 9]{fe90}. It states (in the unweighted
version) that one has
\begin{equation*}\label{HYGenAmalg}
{\mathcal F}
(W({\mathcal F}L^p, \ell^q) )
\hookrightarrow
W({\mathcal F}L^q, \ell^p)
 \quad \mbox{for} \,\, 
  1 \leq q \leq p \leq \infty.
\end{equation*}
In particular it implies the Fourier invariance of the space
$W({\mathcal F}L^p,\ell^p)$, for $1 \leq p \leq \infty$ (the unweighted version
of Theorem 6 of \cite{fe90}). 

\begin{lemma} \label{LocFLp}
For $1 \leq p \leq 2$ we have the chain of continuous embeddings
\begin{equation*}\label{FLpemb}
W({\mathcal F}L^p, \ell^p) \hookrightarrow {\mathcal F}L^p \hookrightarrow 
W({\mathcal F} L^p, \ell^{p^\prime})
\end{equation*}
or by taking inverse Fourier transforms
\begin{equation*}\label{MppP}
M^{p,p} ({\mathbf R}^n) \hookrightarrow
L^p({\mathbf R}^n)
\hookrightarrow
M^{p,p^\prime} ({\mathbf R}^n).
\end{equation*}
In particular, we have
$M^{2,2} ({\mathbf R}^n) = L^2 ({\mathbf R}^n)$.

\end{lemma}

\begin{proof}
We start with the observation that the classical Hausdorff-Young
  estimate implies for $1 \leq p \leq 2$, with $1/p'+1/p =1$: 
   \begin{equation*} \label{HausYoung24A}
 {\mathcal F}L^p \hookrightarrow L^{p^\prime}
\quad \mbox{or equivalently} \quad 
L^p \hookrightarrow {\mathcal F}L^{p^\prime}
 \end{equation*}
 and consequently 
    \begin{equation*} \label{HausYoung24E}
W({\mathcal F} L^p, \ell^p)
\hookrightarrow 
W(L^{p^\prime}, \ell^p)
\hookrightarrow W(L^p,\ell^p) =L^p,
 \end{equation*}
 which implies: 
    \begin{equation*} \label{HausYoung24F}
W({\mathcal F}L^p,\ell^p)
={\mathcal F} (W({\mathcal F} L^p,\ell^p ))
\hookrightarrow {\mathcal F}L^p,
\quad \mbox{for} \,\, 1 \leq p \leq 2.
 \end{equation*}
 Thus the first continuous embeddings is verified.
  For the second inclusion recall that 
\cite{fe90}) (see above, due to the obvious fact that  $p \leq 2 \leq p'$) implies 
    \begin{equation*} \label{HausYoung24C}
{\mathcal F}L^p =
{\mathcal F} (W(L^p, \ell^p)) 
\hookrightarrow {\mathcal F} (W({\mathcal F} L^{p^\prime} , \ell^p  )) 
\hookrightarrow_{HY} 
W({\mathcal F} L^p, \ell^{p^\prime}).
 \end{equation*}

\end{proof}

\subsection{Fourier-Beurling algebra}
\label{subsection Fourier-Beurling}
For $s \geq 0$ the Fourier-Beurling algebra
${\mathcal F}L_s^1({\mathbf R}^n) = {\mathcal F}L_s^1$ is
the set of all
$f   \in   {\mathcal S}^\prime ({\mathbf R}^n)$  such that the norm
$$
\|  f  \|_{{\mathcal F}L_s^1}=
\int_{{\mathbf R}^n}
\langle \xi \rangle^s |  \widehat{f} (\xi)|    d \xi
$$
is finite.
It is well-known that ${\mathcal F}L^1_s ({\mathbf R}^n)$ is a multiplication algebra,
because Beurling algebras are Banach convolution
algebras (\cite{Reiter-Stegeman 2000}). 
We also recall the following result
(e.g., {\cite[Proposition 1.6.14]{Reiter-Stegeman 2000}}):
\begin{lemma}
	\label{approximate units in FL1s}
Let $s \geq 0$ and $f \in {\mathcal F}L^1_s ({\mathbf R}^n)  $.
Then for any $\varepsilon >0$, there exists $\phi \in C^\infty_c ({\mathbf R}^n)$ 	such that
	$ \| f - \phi f \|_{{\mathcal F}L^1_s} < \varepsilon 	$.	
\end{lemma}

For the inclusion relation between ${\mathcal F}L^1_s({\mathbf R}^n)$
and $M^{p,1}_s({\mathbf R}^n)$, we have the following
(cf. \cite{Lu}, \cite{Okoudjou}).

\begin{lemma}
\label{inclusion Mp1s subset FL1}
For $1 \leq p \leq 2$ and $s \geq 0$ one has
$M^{p,1}_s ({\mathbf R}^n) \hookrightarrow {\mathcal F}L^1_s({\mathbf R}^n)$.
\end{lemma}
\begin{proof}
Let $f \in M^{p,1}_s({\mathbf R}^n)$ and  $\varphi \in C^\infty_c({\mathbf R}^n)$
be such that
${\rm supp~} \varphi \subset [-1,1]^n$
and $\sum_{k \in {\mathbf Z}^n} \varphi(\xi - k) =1$.
Since ${\rm supp~} \varphi (\cdot -k) \subset
k+ [-1,1]^n$,
it follows from
 the Minkowski inequality for integral,
 the H\"older inequality and
the Hausdorff-Young inequality that
\begin{align*}
\| f \|_{{\mathcal F}L^1_s}
&\leq \Big\|
\sum_{k \in {\mathbf Z}^n}  \langle \cdot \rangle^s |  \varphi(\cdot - k) \widehat{f} (\cdot)|
\Big\|_{L^1}
\lesssim
\sum_{k \in {\mathbf Z}^n}  \langle k \rangle^s
\| \varphi (\cdot -k ) \widehat{f} (\cdot)
\|_{L^1} \\
&\lesssim
\sum_{k \in {\mathbf Z}^n}  \langle k \rangle^s
\| \varphi (\cdot -k ) \widehat{f} (\cdot)
\|_{L^{p^\prime}}
\lesssim
\sum_{k \in {\mathbf Z}^n}  \langle k \rangle^s
\| \varphi (D -k ) f
\|_{L^p}
= \| f \|_{M^{p,1}_s},
\end{align*}
which yields the desired result.
\end{proof}

\begin{lemma}
\label{inclusion L_0}
For  $s \geq 0$ we define the space $({\mathcal F}L^1_s)_c$ by
$$
({\mathcal F}L^1_s)_c= \{ f \in {\mathcal F}L^1_s ({\mathbf R}^n) ~|~
{\rm supp~} f ~ {\rm is~compact} \}.
$$
Then we have
$({\mathcal F}L^1_s)_c \hookrightarrow M^{p,1}_s({\mathbf R}^n)$
for $1 \leq p < \infty$.
\end{lemma}
\begin{proof}
Let  $f  \in {\mathcal F}L^1_s({\mathbf R}^n)$ and
${\rm supp~}f$ be compact.
Then for  $\phi \in {\mathcal S}({\mathbf R}^n)$ with
$\phi(x)=1$ on ${\rm supp~} f$, we have by Lemma \ref{product estimate}
\begin{align*}
\| f \|_{M^{p,1}_s}  =
\| f \phi \|_{M^{p,1}_s}
\lesssim \| f \|_{M^{\infty,1 }_s} \| \phi \|_{M^{p,1}_s}.
\end{align*}
Moreover,  for $\varphi \in {\mathcal S}({\mathbf R}^n)$
with $\sum_{k \in {\mathbf Z}^n} \varphi (\xi -k)=1$,
we have that
\begin{align*}
\| f \|_{M^{\infty,1}_s}
=
\sum_{k \in {\mathbf Z}^n}
\langle k \rangle^s \| \varphi(D-k) f \|_{L^\infty}
\lesssim
\sum_{k \in {\mathbf Z}^n}
\langle k \rangle^s \|  \varphi (\cdot -k)  \widehat{f} \|_{L^1}
\lesssim \| f \|_{{\mathcal F}L^1_s},
\end{align*}
which implies the desired result.
\end{proof}


\section{Approximate units}
\label{section Approximate units}

In this section, we prove the following result,
which corresponds to the $M^{p,q}_s$-version of
Bhimani-Ratnakumar \cite[Proposition 3.14]{Bhimani Ratnakumar}.
\begin{lemma}
\label{f-psif}
Given $1 \leq p, q < \infty$
and  $s> {n}/{q^\prime}$, or
$q=1$ and  $s \geq 0$. Then for any $f \in M_s^{p,q}({\mathbf R}^n)$,
and  $\varepsilon>0$ there exists
$\phi \in C_c^\infty ({\mathbf  R}^n)$ such that
$$
\| f-\phi f \|_{M_s^{p,q}} < \varepsilon.
$$
\end{lemma}

We remark that Bhimani
\cite[Proposition 4.8]{Bhimani}
considers the case  $q=1$ and $s=0$.

\subsection{Technical lemmas}

To prove Theorem \ref{f-psif} we prepare a lemma.
%
\begin{lemma}
\label{approx M11s}
Let  $s \geq 0$ and
$\psi \in C^\infty_c  ({\mathbf R}^n)$ be such that
$\psi(0)=1$.  For $0< \lambda <1$,
we define $\psi_\lambda  (x)=\psi( \lambda x)$.
Then, for any $g \in {\mathcal S}({\mathbf R}^n)$ and
$\varepsilon>0$, there exists $0< \lambda_0 <1$ such that
$$
\| (1-  \psi_\lambda )g \|_{M^{1,1}_s}  <  \varepsilon
\quad \forall 0< \lambda < \lambda_0. $$
\end{lemma}

\begin{proof}
Let $\phi \in {\mathcal S}({\mathbf R}^n) \setminus \{ 0 \}$.
We first  note that
\begin{align*}
V_\phi( (1-\psi_\lambda )g )(x, \xi)
&=
(2 \pi)^{-n}
e^{- i x \xi} V_{\widehat{\phi}}
\big(\widehat{g}- (2 \pi)^{-n}
(\widehat{\psi_\lambda} *  \widehat{g} )
\big)(\xi, -x)
\end{align*}
by \eqref{basic STFT} in Section \ref{STFT}.
Moreover,  since
$\widehat{ \psi_\lambda } (\eta) = \frac{1}{\lambda^n} \widehat{\psi} ( \frac{\eta}{\lambda} ) $
and
$1= \psi(0) =\frac{1}{(2 \pi)^n} \int_{{\mathbf R}^n} \widehat{\psi_\lambda} (\eta) d \eta$,
we have
\begin{align*}
\widehat{g} (t) -  (2 \pi)^{-n} ( \widehat{\psi_\lambda} * \widehat{g}) (t)
&
=
\frac{1}{ (2 \pi \lambda)^n}
\int_{{\mathbf R}^n}
( \widehat{g} (t) - \widehat{g}  (t- \eta) )
\widehat{\psi} \Big( \frac{\eta}{\lambda}  \Big) d \eta \\
&=
\frac{1}{(2 \pi)^n}
\int_{{\mathbf R}^n }
( \widehat{g} \big(t)  - (T_{ \lambda \eta} \widehat{g})  (t) \big) \widehat{\psi} (\eta) d \eta
\end{align*}
and thus
$$
V_{\widehat{\phi}}  (\widehat{g}-
(2 \pi)^{-n}
(\widehat{\psi_\lambda} *  \widehat{g} ) )(\xi, -x)
=
\frac{1}{(2 \pi)^{2n}}
\int_{{\mathbf R}^n }
V_{\widehat{\phi}} ( \widehat{g} - T_{\lambda \eta} \widehat{g} )(\xi,-x)
 \widehat{\psi} (\eta) d \eta.
$$
Therefore we obtain by the Fubini Theorem that
\begin{align*}
\| (1- \psi_\lambda ) g \|_{M^{1,1}_s}
&  \approx
\iint_{{\mathbf R}^{2n}}
\langle \xi \rangle^s
|V_{ \widehat{\phi} } (\widehat{g}  - (2 \pi)^{-n}
(\widehat{\psi_\lambda} * \widehat{g}   ) ) (\xi, -x) |
 dx d \xi \\
& \lesssim
\int_{{\mathbf R}^n}
\| \langle \xi \rangle^s
V_{ \widehat{\phi} } ( \widehat{g} - T_{\lambda \eta} \widehat{g}   )
(\xi,-x) \|_{L^1 ({\mathbf R}^{2n}_{(x,\xi)})}
|\widehat{\psi} (\eta)| d \eta.
\end{align*}
On the other hand, since  $0< \lambda <1$, $s \geq 0$ and
$$
V_{ \widehat{\phi} } (T_{\lambda \eta} \widehat{g} ) (\xi, -x)
= e^{-i \lambda \eta (-x) } (V_{ \widehat{\phi} }   \widehat{g} ) (\xi - \lambda \eta,-x),
$$
we have
\begin{align*}
&
\| \langle \xi \rangle^s
V_{ \widehat{\phi} } ( \widehat{g} - T_{\lambda \eta} \widehat{g} )
(\xi,-x) \|_{L^1 ({\mathbf R}^{2n}_{(x,\xi)})} \\
& \lesssim
\iint_{{\mathbf R}^{2n}}
\langle \xi \rangle^s
| V_{ \widehat{\phi} } \widehat{g} (\xi, -x)  | d x d \xi
+
\langle \lambda \eta \rangle^{s}
\iint_{{\mathbf R}^{2n}}
\langle \xi - \lambda  \eta \rangle^s
| V_{ \widehat{\phi} } \widehat{g} (\xi - \lambda \eta,-x)  | d x d \xi  \\
& \lesssim
\langle \eta \rangle^s  \| \langle \xi \rangle^s V_{ \widehat{\phi} }
\widehat{g} (\xi, -x) \|_{L^1({\mathbf R}^{2n}_{(x, \xi)} ) }.
\end{align*}
We note that since $\widehat{\phi}, \widehat{g} \in {\mathcal S} ({\mathbf R}^n) $,
we have
$V_{ \widehat{\phi}  } \widehat{g} \in {\mathcal S} ({\mathbf R}^{2n}) $,
and also
\begin{align*}
V_{ \widehat{\phi}} ( \widehat{g} - T_{\lambda \eta}\widehat{g}    )(\xi,-x)
&=
\int_{{\mathbf R}^n}
(\widehat{g}(t) - \widehat{g} (t-\lambda \eta))  \overline{  \widehat{\phi} (t-\xi)} e^{ ixt}dt
\to
0  \quad (\lambda \to 0)
\end{align*}
for all $\eta, \xi, x \in {\mathbf R}^n$.
Thus it follows from
the Lebesgue convergence theorem that
\begin{align*}
\| (1- \psi_\lambda ) g \|_{M^{1,1}_s}
\lesssim
\iiint_{{\mathbf R}^{3n}}
\langle \xi \rangle^s
|V_{ \widehat{\phi}} (\widehat{g} - T_{ \lambda \eta}\widehat{g})(\xi,-x)|
|\widehat{\psi  } (\eta) |
dx d \xi d \eta
\to 0
\end{align*}
as $\lambda \to 0$,
which implies the desired result.

\end{proof}

\subsection{The proof of Theorem \ref{f-psif}}

Let  $\varepsilon >0$.
Since ${\mathcal S} ({\mathbf R}^n)$ is dense in $M_s^{p,q}({\mathbf R}^n)$,
there exists $g  \in  {\mathcal S}({\mathbf R}^n)$ such that
$ \|  f-g   \|_{M_s^{p,q}}   <\varepsilon $.
Moreover  by Lemma \ref{approx M11s},
there exist  $\psi  \in   C_c^\infty
({\mathbf R}^n)$ and $0< \lambda_0 <1$ such that
$
\|   (1-\psi_\lambda)g  \|_{M_s^{1,1}}  <  \varepsilon
$
for all $0< \lambda < \lambda_0 $,
where $\psi_\lambda (x) = \psi (\lambda x)$.
Thus we have  by  Lemma
 \ref{basic pro mod} (v)
and Lemma \ref{product estimate}  that
\begin{align*}
\|  f-  \psi_\lambda  f  \|_{M_s^{p,q}}
&  \lesssim
 \|  f-g \|_{M_s^{p,q}}
+ \| \psi_\lambda  (f-g)  \|_{M_s^{p,q}}
+  \|   (1-\psi_\lambda)g  \|_{M_s^{p,q}}\\
&  \lesssim
(1  +  \|   \psi_\lambda \|_{M^{\infty,q}_s} )  \| f-g  \|_{M_s^{p,q}}
+  \|  (1-\psi_\lambda)g  \|_{M_s^{p,q}}  \\
& \lesssim
(1  +  \|    \psi_\lambda  \|_{M_s^{\infty,1}} )
\|  f-g  \|_{M_s^{p,q}}  +    \|  (1-\psi_\lambda)g  \|_{M_s^{1,1}} \\
& \leq
(2  +   C  \|    \psi \|_{M_s^{\infty,1}} ) \varepsilon,
\end{align*}
which implies the desired result.

\section{Closed ideals in $M^{p,1}_s({\mathbf R}^n)$}
\label{closed ideals in Mp1s}

In this section, we consider the  closed ideals
in  $M_s^{p,1}({\bf R}^n)$,
which play an important role in
Section \ref{Set of spectral synthesis in Mp1s}.
Throughout this section, $X$ stands for $M^{p,1}_s({\mathbf R}^n)$
($1 \leq p< \infty$, $s \geq 0$)
or ${\mathcal F}L^1_s ({\mathbf R}^n)$
$(s \geq 0)$.

\begin{definition}
Let $I$ be a linear subspace of $X$.
Then $I$ is called an {\it ideal} in $X$
if $fg \in I$ whenever $f \in X$ and $g \in I$.
Moreover, if an ideal $I$ in $X$ is a closed subset of
$X$, then $I$ is called a {\it closed ideal} in $X$.
For a subset $S$ of $X$,
the set $\bigcap_{\lambda \in \Lambda} I_\lambda$
is called  the ideal generated by $S$, where
$\{ I_\lambda \}_{\lambda \in \Lambda}$
denoted the set of all ideals in $X$ containing $S$.
\end{definition}

It is easy to see that the closed ideal generated by $S$
coincides with the closure in $X$ of the set
$$
\Big\{ \sum^n_{j=1} f_j g_j ~\Big|~
f_j \in X, \ g_j \in S , \ n \in {\mathbf N}
\Big\}.
$$
\begin{definition}
For  a closed ideal $I$ of $X$  the zero-set of $I$
is defined by $\displaystyle Z(I)=\bigcap_{f\in I}f^{-1}( \{ 0 \} )$
with $f^{-1}( \{ 0 \} ) = \{ x \in {\mathbf R}^2 ~|~ f(x) = 0\}$.
\end{definition}

We note that $x \in Z(I)$ if and only if
$f(x) =0$ for all $f \in I$.
Moreover,  $Z(I)$ is a closed subset of $X$
whenever $I$ is a closed ideal in $X$.
In fact, if $f \in X$, then
$f$ is continuous on ${\mathbf R}^n$ and thus
$f^{-1} (\{ 0 \})$ is a closed subset of ${\mathbf R}^n$.

The following lemma seems to be known to many people,
but for the reader's convenience, we  give the
proof. We will write $f|_E =0$ if $f(x) =0$ for all $x \in E$.
\begin{lemma}
	\label{Basic lemma 1 closed ideal}
	Let $E$ be a closed subset of ${\mathbf R}^n$.
	Then
	$$
	I(E)= 	\{ f \in X ~|~ f |_E = 0 \}
	$$
	is a closed ideal in $X$ with $E = Z(I(E))$.
	\end{lemma}
\begin{proof}
	We give the proof  only for the case $X= M^{p,1}_s({\mathbf R}^n)$;
	the other case is similar.
	It is clear that $I(E)$ is an ideal in $M^{p,1}_s ({\mathbf R}^n)$.
	To see that $I(E)$ is closed,
	let $f \in M^{p,1}_s ({\mathbf R}^n)$, $\{ f_n \}^\infty_{n=1} \subset I(E)$
	and
	$ \| f_n - f \|_{M^{p,1}_s} \to 0$ $(n \to \infty)$.
	Since
	$$
	\| f_n -f \|_{L^\infty}
	\lesssim
	\| f_n -f \|_{M^{\infty,1}}
	\lesssim \| f_n -f \|_{M^{p,1}}
	\lesssim
	\| f_n - f \|_{M^{p,1}_s},
	$$
	we see that
	$\{ f_n \}^\infty_{n=1}$  converges pointwise to $f$ on ${\mathbf R}^n$.
	Since $f_n |_E =0$, we have $f|_E=0$, and thus $f \in I(E)$.
	Hence  $I(E)$ is closed.
	Next we prove $E=Z(I(E))$.
	Since $E \subset Z(I(E))$ is clear, we
	show $Z(I(E)) \subset E$.
	Suppose
	$x_0  \not\in E$.
	Since $E$ is closed
	and  $C^\infty_c  ({\mathbf R}^n) \subset M^{p,1}_s ({\mathbf R}^n) $,
	there exists
	$f \in M^{p,1}_s ({\mathbf R}^n)$ such that
	$f(x_0 ) =1$ and $f |_E =0$.
	Then $f \in I(E)$ and $f(x_0) \not= 0 $.
	Thus   $x_0 \not\in Z(I(E))$,
	which implies the desired result. 	
\end{proof}

We  also prepare the following lemma,
which also characterises
closed ideals in $M^{p,1}_s ({\mathbf R}^n)$.
We will  call  a subspace $M$ of $M^{p,1}_s({\mathbf R}^n)$
is modulation invariant if
$M_\eta f(x) = e^{ix \eta} f(x) \in M$ whenever
	$f \in M$ and $\eta \in {\mathbf R}^n$.

\begin{lemma}
\label{characterization of ideal in Mp1s}
Given $1 \leq p < \infty$ and  $s \geq 0$. \newline
Then for any closed ideal $I$  in $M^{p,1}_s ({\mathbf R}^n)$, one has$:$\\
$(i)$
If $g \in M^{p^\prime , \infty  }_{-s}  ({\mathbf R}^n) $ satisfies
	\begin{align}
		\label{equiv1}
		\iint_{{\mathbf R}^{2n}}
		V_\phi f(x, \xi)
		\overline{ V_\psi g(x,\xi)  } dx d \xi =0
\quad {\rm for~ all~} f \in I
	\end{align}
for any pair  $\phi, \psi \in {\mathcal S} ({\mathbf R}^n)$
with $\langle \psi, \phi \rangle  \not= 0$, if and only if 
	\begin{align}
		\label{equiv2}
		\int_{{\mathbf R}^n}
		\Big[
		V_\phi f(x,\cdot) * V_{ \overline{\psi}  } \overline{g} (x, \cdot) \Big]
		(\xi) dx =0
\quad {\rm for~ all~} f \in I,  \xi \in {\mathbf R}^n
	\end{align}
for  any pair  $\phi, \psi \in {\mathcal S} ({\mathbf R}^n)$ with $\langle \psi, \phi \rangle  \not= 0$.\\
$(ii)$	Every closed ideal in
$M^{p,1}_s({\mathbf R}^n)$
coincides with a closed modulation invariant
	subspace of $M^{p,1}_s ({\mathbf R}^n)$.
	
\end{lemma}

\begin{proof}
We first assume that  \eqref{equiv2} holds.
Then we have
\begin{align*}
0
&=
\int_{{\mathbf R}^n}
\Big[
V_\phi f(x,\cdot) * V_{ \overline{\psi}  } \overline{g} (x, \cdot) \Big]
(\xi) dx
=
(2 \pi)^n
\langle \psi, \phi \rangle
\langle M_{- \xi} f, g \rangle
\end{align*}
for all $f \in I$,
$\xi \in {\mathbf R}^n$
and $\phi, \psi \in {\mathcal S} ({\mathbf R}^n)$
with $\langle \psi, \phi \rangle  \not= 0$,
and thus $\langle f, g \rangle=0$.
Hence
$$
\iint_{{\mathbf R}^{2n}}
V_\phi f(x, \xi)
\overline{ V_\psi g(x,\xi)  } dx d \xi
= (2 \pi)^{-n}
\langle \psi, \phi \rangle
\langle f, g \rangle =0,
$$
which gives
\eqref{equiv1}.

Next we assume that \eqref{equiv1} holds.
We note that  Lemma \ref{product estimate} implies that
$f \overline{g} \in M^{1, \infty}_{-s} ({\mathbf R}^n)  \hookrightarrow
M^{p^\prime, \infty}_{-s} ({\mathbf R}^n)$
for all $f \in I$.
Then  for  all  $h \in M^{p,1}_s ({\mathbf R}^n)$
and $\phi, \psi \in {\mathcal S}({\mathbf  R}^n)$
with $\langle \psi, \phi \rangle \not= 0$,
we have
\begin{align*}
\langle \psi, \phi  \rangle
\langle f \overline{g} ,h \rangle
&=
\langle \psi, \phi  \rangle
\langle f \overline{h}, g \rangle
=
(2 \pi)^{-n}
\langle
V_{ \phi  } (f \overline{h}) ,
V_\psi g \rangle.
\end{align*}
On the other hand, since $f \in I$ and $I$ is a closed ideal in $M^{p,1}_s({\mathbf R}^n)$,
we see that  $f \overline{h} \in I$.
Therefore, \eqref{equiv1} implies that
$
\langle
V_{ \phi   } (f \overline{h}) ,
V_\psi g \rangle =0,
$
and thus $ \langle f \overline{g},h \rangle =0$.
By the duality, we obtain
$f \overline{g}=0$.
Hence, we obtain for $x, \xi \in {\mathbf R}^n$
that
\begin{align*}
0&=
V_{\phi \overline{\psi}} (f  \overline{g}) (x,\xi) \\
&=
(2 \pi)^{-n}
\int_{{\mathbf R}^n}
V_{\phi} f (x, \xi - \eta)  V_{\overline{\psi}} \overline{g}(x, \eta) d \eta \\
&=
(2 \pi)^{-n}
\int_{{\mathbf R}^n}
V_{\phi} (M_{- \xi} f ) (x,  - \eta)
\overline{ V_\psi g(x, - \eta)   }
 d \eta,
\end{align*}
and thus
$\langle \psi, \phi \rangle \langle M_{- \xi} f, g \rangle =0$,
which implies
\eqref{equiv2}.

Finally, we give the characterization of  closed
ideals in $M^{p,1}_s({\mathbf R}^n)$.
Let $I$ be a closed ideal in $M^{p,1}_s({\mathbf R}^n)$.
We first note  that if  $g \in M^{p^\prime, \infty}_{-s} ({\mathbf R}^n)$
satisfies  \eqref{equiv2},
then
\begin{align}
\label{equiv3}
0
&=\int_{{\mathbf R}^n}
[ V_\phi f (x, \cdot)  * V_{\overline{\psi}} \overline{g}  (x, \cdot) ] (\xi) dx
=
\langle V_\phi ( M_{- \xi} f),  V_\psi  g \rangle
\end{align}
holds for all $f \in I$, $\xi \in {\mathbf R}^n$ and
$\phi, \psi \in {\mathcal S} ({\mathbf R}^n)$
with $\langle \psi, \phi \rangle  \not= 0$.
Suppose that  $I$ is not modulation invariant.
Then there exist $f_0 \in I$ and $\eta_0 \in {\mathbf R}^n$
such that $M_{- \eta_0} f_0 \not\in I $.
On the other hand, we note that
$(M^{p,1}_s ({\mathbf R}^n))^\prime = M^{p^\prime, \infty}_{-s} ({\mathbf R}^n)$,
i.e., for any $\ell \in (M^{p,1}_s ({\mathbf R}^n))^\prime$
and $\phi, \psi \in {\mathcal S}({\mathbf R}^n)$
with $\langle \psi, \phi \rangle=0$, there
exists $g \in M^{p^\prime, \infty}_{-s } ({\mathbf R}^n)$
such that
$$
\ell (f)
= \frac{1}{  (2 \pi)^n \langle \psi, \phi \rangle}
\langle V_\phi f , V_\psi g \rangle.
$$
Therefore, the Hahn-Banach theorem
and the duality imply that there exists
$g_0 \in M^{p^\prime,\infty}_{-s} ({\mathbf R}^n)$
such that
$$
\langle f, g_0 \rangle
=
\frac{1}
{ (2 \pi)^n \langle \psi, \phi \rangle }
\langle V_\phi f , V_\psi g_0  \rangle
 =0
$$
for all $f \in I$, and
\begin{align*}
\langle M_{- \eta} f,  g_0 \rangle
= \frac{1}
{ (2 \pi)^n \langle \psi, \phi \rangle }
\langle V_\phi (M_{- \eta} f_0), V_\psi g_0 \rangle
=1.
\end{align*}
The above two equalities imply that $g_0$ satisfies \eqref{equiv2}
but not \eqref{equiv3}, which  yields a contradiction.

Conversely, we assume that $I$ is a closed
modulation invariant subspace
of $M^{p,1}_s ({\mathbf R}^n)$.
If $I$ is not an ideal in $M^{p,1}_s ({\mathbf R}^n)$,
then there exist $f_0 \in I$
and $h_0 \in M^{p,1}_s ({\mathbf R}^n)$ such that
$f_0 \overline{h_0} \not\in I$.
The Hahn-Banach theorem and the duality imply that
there exists $g_0 \in M^{p^\prime, \infty}_{-s} ({\mathbf R}^n)$
such that
\begin{align}
\label{equiv4}
\langle f,g_0 \rangle
=\frac{1}{ (2 \pi)^n
\langle \psi, \phi \overline{\phi} \rangle}
\langle V_{\phi \overline{\phi}} f, V_\psi g_0 \rangle
=0
\end{align}
for all $f \in I$, and
\begin{align}
\label{equiv5}
\langle f_0 \overline{h_0}, g_0 \rangle
=
\frac{1}{(2 \pi)^n \langle \psi, \phi \overline{\phi} \rangle}
\langle  V_{\phi \overline{\phi}} (f_0 \overline{h_0}), V_\psi g_0
\rangle
=1.
\end{align}
We note that
$$
V_{\phi \overline{\phi}} (f_0 \overline{h_0}) (x,\xi)
= (2 \pi)^{-n}  [V_\phi f_0 (x, \cdot) * V_{ \overline{\phi}  }  \overline{h_0} (x, \cdot)   ] (\xi)
$$
and thus
\begin{align*}
\langle V_{\phi \overline{\phi}} (f_0 \overline{h_0}) , V_\psi g_0 \rangle
&=
(2 \pi)^{-n}
\int_{{\mathbf R}^n}
[ (V_{\phi} f_0 (x,\cdot) * V_{\overline{\phi}}  \overline{h_0}  (x,\cdot) )* V_{ \overline{\psi} }  \overline{g_0}
(x,\cdot)
] (0)dx  \\
&=(2 \pi)^{-n}
\int_{{\mathbf R}^n}
[ (V_{\phi} f_0 (x,\cdot) *
V_{ \overline{\psi} }  \overline{g_0}
  (x,\cdot) )
*  V_{\overline{\phi}}  \overline{h_0}
(x,\cdot)
] (0)dx.
\end{align*}
Moreover
for all $x, \xi \in {\mathbf R}^n$,
we have by the Parseval identity that
\begin{align*}
&[V_\phi f_0 (x, \cdot) * V_{ \overline{\psi}} \overline{g_0} (x, \cdot) ] (\xi) \\
&=
\int_{{\mathbf R}^n}
V_\phi (M_{- \xi} f_0) (x,  \eta) \overline{ V_\psi g_0 (x,  \eta)  } d \eta \\
&=
\int_{{\mathbf R}^n}
{\mathcal F}_{y \to \eta} [ (M_{- \xi} f_0 ) (y)\cdot  \overline{\phi (y-x) } ] (\eta)
\overline{
{\mathcal F}_{y \to \eta } [ g_0(y) \overline{ \psi (y- x) } ] (\eta)
} d \eta \\
&=
(2 \pi)^n
\int_{{\mathbf R}^n}
(M_{- \xi} f_0 ) (y)\cdot  \overline{\phi (y-x) }
\overline{
g_0(y) \overline{ \psi (y- x) }
} d y \\
&=
\langle M_{-\xi} f_0 \cdot  T_x (\overline{ \phi } \psi) , g_0 \rangle.
\end{align*}
Since $I$ is modulation invariant, we have
$M_{- \xi} f_0 \in I $, and thus
$ M_{-\xi} f_0 \cdot  T_x (\overline{ \phi } \psi)  \in I$
by $ T_x (\overline{\phi} \psi ) \in M^{p,1}_s ({\mathbf R}^n) $.
Therefore we have by \eqref{equiv4}
$$
[V_\phi f_0 (x, \cdot) * V_{ \overline{\psi}} \overline{g_0} (x, \cdot) ] (\xi) =0.
$$
Hence
$
\langle  V_{\phi \overline{\phi}} (f_0 \overline{h_0}), V_\psi g_0
\rangle
=0$,
which contradicts
\eqref{equiv5}.

\end{proof}


\section{Wiener-L\'evy
theorem for  $M_s^{p,1}({\mathbf R}^n)$}
\label{Wiener-Levy type theorems}

In this section, we consider the Wiener-L\'evy theorem
for  $M_s^{p,1}({\mathbf R}^n)$
(cf. \cite[Theorem 1.3.1]{Reiter-Stegeman 2000}).
More precisely, we show the following result.

\begin{theorem}
\label{image of analytic function}
Given $1 \leq p < \infty$, $s \geq 0$,
$f \in M^{p,1}_s ({\mathbf R}^n)$ and
a compact subset $K \subset {\mathbf R}^n$.
Suppose that $F$ is an analytic function on a neighborhood of
$f(K) = \{ f(x) ~|~ x \in K \}$ in ${\mathbf C}$.
Then there exists $g \in M^{p,1}_s ({\mathbf R}^n)$ such that
$ g(x) = F(f(x)) $ for all $x \in K$. 	
\end{theorem}

To prove Theorem \ref{image of analytic function},
we prepare the following lemmas.
\begin{lemma}
\label{Wiener-Levy pre-lemma 1}
Let $s \geq 0$,
$f,h \in M^{1,1}_s({\mathbf R}^n)$
and $x_0 \in {\mathbf R}^n$. We set
$$
H^\lambda_{x_0} (x) =
\Big(  f \Big( x_0 + \frac{x}{\lambda}  \Big) - f(x_0) \Big) h(x),\quad
 \mbox{for} \,\,  \lambda > 0.
$$
Then we have
\begin{align}
\label{norm of Hlambda}
\lim_{\lambda \to \infty} \| H^\lambda_{x_0} \|_{M^{1,1}_s} =0.
\end{align}
\end{lemma}

\begin{proof}
With $g_0(t) = e^{ - \frac{|t|^2}{2} }$ 
we have by \eqref{basic STFT} in Section \ref{STFT} for $\lambda>1$
\begin{align*}
V_{g_0} (H^\lambda_{x_0}) (x,\xi)
&=
(2 \pi)^{-n}
e^{i x \xi} V_{\widehat{g_0}} ( H^\lambda_{x_0} )^\wedge (\xi, -x) \\
&=
(2 \pi)^{-n}   \big(   (H^\lambda_{x_0})^\wedge * M_{-x} (\widehat{g_0})^*    \big) (\xi) \\
&= (2 \pi)^{ - \frac{n}{2} }
  \big(   (H^\lambda_{x_0})^\wedge * M_{-x} \widehat{g_0}    \big) (\xi).
\end{align*}
We note that
\begin{align*}
\Big(
f \Big( x_0 + \frac{ \cdot }{\lambda} \Big)
\Big)^\wedge (\eta)
=
(T_{\lambda x_0 } D_\lambda f )^\wedge (\eta)
=
\lambda^n e^{i \lambda x_0 \eta} \widehat{f} (\lambda \eta)
\end{align*}
with $D_\lambda f (\eta) = f (   \frac{\eta}{\lambda}  )$.
Moreover, since $f , \widehat{f}\in  M^{1,1}({\mathbf R}^n)
\hookrightarrow  L^1({\mathbf R}^n) $,
we have
by the  Fourier  inversion formula that
$$
f(x_0) =
\frac{1}{(2 \pi)^n} \int_{{\mathbf R}^n} \lambda^n
 e^{i \lambda x_0 \eta} \widehat{f} (\lambda \eta) d \eta.
$$
Thus, by putting  $H(x, \xi) =  [ \widehat{h} * M_{-x} g_0 ] (\xi)  $,
we see
that
\begin{align*}
&
\big((H^\lambda_{x_0} )^\wedge * M_{-x} g_0  \big)(\xi) \\
&=
\frac{1}{
(2 \pi)^{n}}
\Big[
\Big( f \Big( x_0 + \frac{\cdot}{\lambda} \Big)\Big)^\wedge
* \widehat{h} * M_{-x} g_0 \Big](\xi)
- f(x_0) [ \widehat{h} * M_{-x} g_0 ] (\xi) \\
&=
\frac{1}{(2 \pi)^{n}}
\int_{{\mathbf R}^n}
\lambda^n e^{i \lambda x_0 \eta} \widehat{f} (\lambda \eta)
\Big(
H(x, \xi - \eta)
- H(x,\xi)
\Big) d \eta \\
&=
\frac{1}{(2 \pi)^n}
\int_{{\mathbf R}^n}
e^{ix_0 \eta} \widehat{f} (\eta)
\Big(
H
\Big( x,  \xi - \frac{\eta}{\lambda} \Big)
-  H(x, \xi) \Big) d \eta.
\end{align*}
Therefore, we have by the Fubini Theorem that
\begin{align*}
&\| H^\lambda_{x_0} \|_{M^{1,1}_s} \\
&
\lesssim
\int_{{\mathbf R}^n}
| \widehat{f} (\eta)   |
\Big(
\iint_{{\mathbf R}^{2n}}
\langle \xi \rangle^s
\Big| H \Big(x, \xi - \frac{\eta}{\lambda} \Big)
-H(x, \xi)
\Big| dx d \xi
\Big) d \eta.
\end{align*}
Moreover,  by the submultiplicity of
$\langle \cdot \rangle^s$ for $s \geq 0$,
one has
$$
\langle \xi \rangle^s
\lesssim
\Big\langle \xi - \frac{\eta}{\lambda}
\Big\rangle^s \Big\langle \frac{\eta}{\lambda}
\Big\rangle^s
\quad
and
\quad
\Big\langle \frac{\eta}{\lambda}
\Big\rangle
\approx
1 + \Big| \frac{\eta}{\lambda} \Big|
\leq
1+ |\eta|
\approx \langle \eta \rangle
$$
for large $\lambda$,
and thus
\begin{align*}
| \widehat{f} (\eta)  |
 \iint_{{\mathbf R}^{2n}}
\langle \xi \rangle^s
\Big| H\Big(x, \xi - \frac{\eta}{\lambda} \Big)
- H(x, \xi)
\Big|  dx d \xi
 \lesssim \langle \eta \rangle^s
| \widehat{f} (\eta)  |
\| h \|_{M^{1,1}_s}.
\end{align*}
Assume now that \eqref{norm of Hlambda}
is not valid, i.e.,
there exists $\varepsilon_0>0$
such that for some sequence $\{ \lambda_n \}^\infty_{n=1}$
with $\lim_{n \to \infty} \lambda_n = \infty$,
one has
$\| H^{\lambda_n}_{x_0} \|_{M^{1,1}_s} \geq \varepsilon_0$.
But this is not possible since the Lebesgue dominated convergence theorem
implies that $\| H^{\lambda_n}_{x_0} \|_{M^{1,1}_s}  \to 0$ $(n \to \infty)$.
This implies  \eqref{norm of Hlambda}.
\end{proof}
\begin{lemma}
\label{Wiener-Levy pre-lemma 2}
Given $1 \leq p < \infty$, $s \geq 0$, $f \in M^{p,1}_s ({\mathbf R}^n)$,
$\tau \in C^\infty_c({\mathbf R}^n)$
and $x_0 \in {\mathbf R}^n$.
For $\lambda >0$, we define
$$
G^\lambda_{x_0} (x)
=
\Big( f \Big( x_0 + \frac{x}{\lambda} \Big) - f(x_0)
\Big) \tau  (x).
$$
Then we have
$$
\lim_{\lambda \to \infty}
\|  G^\lambda_{x_0} \|_{M^{p,1}_s} =0.
$$
\end{lemma}

\begin{proof}
Suppose that  $\lambda >0$ and ${\rm supp~} \tau \subset B_\lambda (0)$.
Take $\psi \in C^\infty_c ({\mathbf R}^n)$
such that  $\psi (x)=1$ on $B_2(x_0)$.
Since $\psi(x_0)=1$ and
 $\psi (x_0 + \frac{x}{\lambda}) =1$ on ${\rm supp~} \tau$,
we see that
$$
\Big( f \Big( x_0 + \frac{x}{\lambda} \Big) - f(x_0)
\Big) \tau (x)
=
\Big( (f \psi ) \Big( x_0 + \frac{x}{\lambda} \Big) - (f \psi)(x_0)
\Big) \tau (x)
$$
for all $x \in {\mathbf R}^n$.
On the other hand,
Lemma \ref{product estimate}  implies that
$$
\| f \psi \|_{M^{1,1}_s} \lesssim
\| f \|_{M^{\infty,1}_s} \| \psi \|_{M^{1,1}_s}
\lesssim
\| f \|_{M^{p,1}_s} \| \psi \|_{M^{1,1}_s},
$$
and thus $f \psi \in M^{1,1}_s({\mathbf R}^n)$.
Hence, it follows from  Lemma \ref{Wiener-Levy pre-lemma 1} that
\begin{align*}
\| G^\lambda_{x_0}      \|_{M^{p,1}_s}
\lesssim
\| G^\lambda_{x_0}      \|_{M^{1,1}_s}
=
\Big\| \Big( (f  \psi) \Big( x_0 + \frac{\cdot}{\lambda} \Big) - (f \psi)(x_0)
\Big) \tau \Big\|_{M^{1,1}_s}
\to 0
\end{align*}
as $\lambda \to \infty$.
\end{proof}
\begin{lemma}
\label{local image of analytic function}
Let $f$, $K$ and $F$ be as in Theorem \ref{image of analytic function},
and  $x_0 \in K$.
Then there exists $g_{x_0} \in M^{p,1}_s ({\mathbf R}^n)$
such that
$g_{x_0} (x) = F(f(x))$
on some neighborhood of $x_0$.
\end{lemma}
\begin{proof}
Set $z_0 = f(x_0)$. Since
$F$ is analytic at $z_0$, there exists
$0< \varepsilon_0 <1$ such that
$F(z)$ has the power series expansion
$$
F(z) = F(z_0) +
\sum^\infty_{j=1} c_j (z-z_0)^j
\quad (|z-z_0| < \varepsilon_0).
$$
Let $\tau \in C_c^\infty ({\mathbf R}^n)$ be a function
which is $1$ near the origin.
Then it follows from Lemma \ref{Wiener-Levy pre-lemma 2}
that there exists $\lambda_{x_0} >0$
 such that
$\|  ( f (x_0 + \frac{\cdot}{\lambda_{x_0}} ) -f(x_0)) \tau  \|_{M^{p,1}_s} < \varepsilon_0 /c$,
where the constant $c$ is as
in Lemma \ref{basic pro mod} $(iv)$.
Thus we see that
$$
b_{x_0}(x)
= F(f(x_0)) \tau (x) +
\sum^\infty_{j=1}  c_j
\Big(  \Big( f  \Big( x_0 + \frac{x}{\lambda_{x_0}} \Big) -f(x_0)   \Big)  \tau(x)  \Big)^j
$$
converges in $M^{p,1}_s ({\mathbf R}^n)$.
Hence
\begin{align*}
g_{x_0} (x) &= b_{x_0} ( \lambda_{x_0} (x-x_0) )  \\
&=
 F(f(x_0)) \tau_{x_0} (x) +
\sum^\infty_{j=1}  c_j
\big(  \big( f  ( x ) -f(x_0)   \big)  \tau_{x_0}(x)  \big)^j
\end{align*}
is in $ M^{p,1}_s ({\mathbf R}^n)$,
where we denote $\tau_{x_0} (x)=  \tau ( \lambda_{x_0} (x-x_0))$.
Moreover, since $g_{x_0}(x) = F(f(x))$
in some neighborhood of $x_0$, we obtain the desired result.
\end{proof}

Now we prove Theorem \ref{image of analytic function}
by extending the local result above to the compact set $K$.
\begin{proof}[The proof of Theorem \ref{image of analytic function}]
We first note that
by Lemma  \ref{local image of analytic function},
we have for each $x_0 \in K$,
there exist $\tau_{x_0} \in C^\infty_c ({\mathbf R}^n)$
and a neighborhood $U_{x_0}$ of $x_0$
such that $g_{x_0}$ defined by
$$
g_{x_0}(x)
= F(f(x_0)) \tau_{x_0} (x) +
\sum^\infty_{j=1}  c_j
\big(( f (x) -f(x_0) )  \tau_{x_0} (x)  \big)^j
$$
satisfies $g_{x_0}(x) = F(f (x) )$ on $U_{x_0}$.
Since
$K \subset \bigcup_{x_0 \in K} U_{x_0}$
and $K$ is compact,
there exist
$\{ x_j \}^N_{j=1} \subset K$
such that the
corresponding neighborhoods $\{ U_{x_j}  \}_{j=1}^N$
cover $K$.
Now we  set
$h_1(x) = \tau_{x_1} (x) $ and
$$
h_j (x) = \tau_{x_j} (x)  (1- \tau_{x_1} (x)) \cdots (1- \tau_{x_{j-1}} (x))
\quad
(j=2,3, \cdots, N).
$$
Since $\tau_{x_j}(x) =1$ and $g_{x_j}(x) = F(f (x))$ on $U_{x_j}$
$(j=1, \cdots,N)$, it is easily see that
$\sum^N_{j=1} h_j (x)=1$ on $K$.
Therefore $g \in M^{p,1}_s ({\mathbf R}^n)$
defined by
 $g(x)= \sum^N_{j=1} h_j  (x) g_{x_j} (x)$
satisfies
\begin{align*}
g(x) =
\sum^N_{j=1} h_j(x) F(f(x))
= F(f(x)) \sum^N_{j=1} h_j(x)
= F(f(x)).
\end{align*}
for all $x \in K$, which yields the desired result.
\end{proof}

Under a mild extra condition on $f$ respectively
$F$, one can extend the local result to the
following global result.
\begin{theorem}
\label{image of analytic function 2}
Let $1 \leq p < \infty$, $s \geq 0$,
$f \in M^{p,1}_s({\mathbf R}^n)$ and
$F$ be analytic on an open  neighborhood of
$f({\mathbf R}^n) \cup \{ 0 \}$ with $F(0)=0$.
Then there exists $g \in M^{p,1}_s({\mathbf R}^n)$
such that $g(x)= F(f(x))$.
\end{theorem}
\begin{proof}
Since $F$ is analytic on a neighborhood of $0$ with $F(0)=0$,
there exists $\varepsilon_0>0$ such that
$F(z)$ has the power series representation
$$
F(z) = \sum^\infty_{j=1} c_j z^j
\quad
(|z| < \varepsilon_0).
$$
It follows from  Theorem \ref{f-psif}
that for any $\varepsilon$ with $0< \varepsilon < \varepsilon_0$,
there exists $\phi \in C^\infty_c({\mathbf R}^n)$ such that
$$
\| f -\phi f \|_{L^\infty}
\lesssim \| f - \phi f \|_{M^{p,1}_s}< \varepsilon / c,
$$
where the constant $c$ is as
in Lemma \ref{basic pro mod} $(iv)$.
Now we set
$$
g_0(x) = \sum^\infty_{j=1} c_j \big( f (x) - \phi (x) f (x)  \big)^j.
$$
Then we have $g_0 \in M^{p,1}_s({\mathbf R}^n)$ and
$g_0(x) = F( f(x) - \phi(x) f(x))   =F(f(x))$ $(x \not\in {\rm supp~}\phi)$.
On the other hand, let $\tau_0 \in C^\infty_c({\mathbf R}^n)$ be
such that $\tau_0 (x) =1$
 on ${\rm supp~} \phi$.
Then it follow from Theorem \ref{image of analytic function}
that there exists exists $g_1 \in M^{p,1}_s({\mathbf R}^n)$
such that
$g_1(x) = F(f(x))$ on  ${\rm supp~} \tau_0$.
Now we set
$$
g(x)=  (1- \tau_0 (x))g_0 (x) + \tau_0 (x)  g_1 (x).
$$
We note that $g \in M^{p,1}_s({\mathbf R}^n)$.
If $x \in {\rm supp~} \phi$, then
$\tau_0(x)=1$ and $g_1(x) =F(f(x))$,
and thus $g(x) =F(f(x))$.
Moreover,  if $x  \in {\rm supp~} \tau_0 \setminus   {\rm supp~}\phi$,
then  $g_0(x)  = F( f(x) - \phi (x) f(x)) = F(f(x))$ and $g_1(x) = F(f(x))$, and thus
$g(x) = F(f(x))$. Finally, if $x \not\in {\rm supp~} \tau_0$, then
$\tau_0 (x)=0$ and $g_0 (x) = F(f(x))$, and thus
$g(x) = F(f(x))$.
Hence we obtain the desired result.
\end{proof}

As an application  of Theorem \ref{image of analytic function},
we obtain the following version of
Wiener's general Tauberian theorem for $M^{p,1}_s({\mathbf R}^n)$
(cf. \cite[Theorem 1.4.1]{Reiter-Stegeman 2000},
\cite[Ch VIII. 6.4]{Katznelson}).
\begin{theorem}
\label{ideal generated by f}
Let $1 \leq p < \infty$, $s \geq 0$,
$f \in M^{p,1}_s({\mathbf R}^n)$
and  $I$ be a closed ideal in $M^{p,1}_s({\mathbf R}^n)$.
If $I$ is generated by one function $f$ in $M^{p,1}_s ({\mathbf R}^n)$,
then we have:  $I  =M^{p,1}_s({\mathbf R}^n) $ if and only if
$f(x) \not= 0$ $(x \in {\mathbf R}^n)$.
\end{theorem}

\begin{proof}
Suppose that $I  =M^{p,1}_s({\mathbf R}^n) $.
Since $I$ is the closed ideal generated by $f \in M^{p,1}_s({\mathbf R}^n)$,
we see that $I$ is equal to the closure  in $M^{p,1}_s({\mathbf R}^n)$
of the set
$$
\Big\{ \sum_{j=1}^N \lambda_j \phi_j  f ~\Big|~
\lambda_j \in {\mathbf C}, \
\phi_j
\in M^{p,1}_s ({\mathbf R}^n), \ N \in {\mathbf N}  \Big\}.
$$
If $f(x_0) = 0$ for some $x_0 \in {\mathbf R}^n$,
then we have  $g(x_0)=0$ for all $g \in I$.
Since ${\mathcal S}({\mathbf R}^n) \subset M^{p,1}_s({\mathbf R}^n)$,
this contradicts $I=M^{p,1}_s({\mathbf R}^n)$.

Conversely, suppose that
$f(x) \not= 0$ for all $x \in {\mathbf R}^n$.
Let $K$ be a compact subset of ${\mathbf R}^n$
and $\phi \in C^\infty_c({\mathbf R}^n)$
be such that ${\rm supp~} \phi \subset K$.
Since
$F(z) = \frac{1}{z}$ is analytic on ${\mathbf C} \setminus \{ 0 \}$
and
$f(x) \not= 0$ for all $x \in K$,
it follows from Theorem \ref{image of analytic function}
that
there exists $g \in M^{p,1}_s({\mathbf R}^n)$
such that $f g =1$ on $K$.
Since $f \in I$ and $I$ is an ideal in
$M^{p,1}_s ({\mathbf R}^n)$, we have
$\phi = \phi g \cdot f \in I$.
Moreover, since
$C^\infty_c({\mathbf R}^n) $
is dense in $M^{p,1}_s ({\mathbf R}^n)$
(cf. Lemma \ref{basic pro mod} and Proposition \ref{f-psif}),
we obtain $I= M^{p,1}_s ({\mathbf R}^n)$.
\end{proof}

\begin{remark}
\label{modulation invariance of I}
By Lemma \ref{characterization of ideal in Mp1s},
the ideal
$I$ in Theorem \ref{ideal generated by f}
 is a closed modulation invariant subspace of
$M^{p,1}_s({\mathbf R}^n)$.
Thus  $I$ is equal to  the closure in $M^{p,1}_s({\mathbf R}^n)$
of the set
$$
A=
\Big\{
\sum^N_{j=1} \lambda_j  e^{i \eta_j x} f(x)
~|~ \lambda_j \in {\mathbf C}, \ \eta_j \in {\mathbf R}^n,
\ N \in  {\mathbf N} \Big\}.
$$
\end{remark}

As a corollary of Theorem \ref{ideal generated by f},
we also obtain
a variant of
Wiener's approximation theorem
for the Wiener amalgam space
$W({\mathcal F}L^p, \ell^1_s) ({\mathbf R}^n)$
(cf. \cite[Theorem 1.4.8]{Reiter-Stegeman 2000}).
\begin{corollary}
Let $1 \leq p < \infty$, $s \geq 0$ and
$W({\mathcal F}L^p, \ell^1_s)  ({\mathbf R}^n)$
be the Wiener amalgam space
consisting  of all $\widehat{f}  \in {\mathcal S}^\prime ({\mathbf R}^n)$
such that $f \in M^{p,1}_s({\mathbf R}^n)$
with the norm
$\| \widehat{f} \|_{ W({\mathcal F}L^p, \ell^1_s)} = \| f \|_{M^{p,1}_s} $.
Then for any $\widehat{f} \in W({\mathcal F}L^p, \ell^1_s) ({\mathbf R}^n)$, we have that
$f(x) \not= 0$ for all $x \in {\mathbf R}^n$ if and only if
the set
$$
{\mathcal F}(A)=
\Big\{
\sum^N_{j=1} \lambda_j \widehat{f} (\xi - \eta_j )
~|~ \lambda_j \in {\mathbf C}, \ \eta_j \in {\mathbf R}^n,
\ N \in  {\mathbf N} \Big\}
$$
is dense in $W({\mathcal F}L^p, \ell^1_s)({\mathbf R}^n)$.
\end{corollary}
\begin{proof}
Let  $\widehat{f} \in W({\mathcal F}L^p, \ell^1_s) ({\mathbf R}^n) $.
We first assume that $f(x) \not= 0$ for all $x \in {\mathbf R}^n$.
Then Remark \ref{modulation invariance of I} implies
that for any $g \in W({\mathcal F}L^p, \ell^1_s) ({\mathbf R}^n)$ and $\varepsilon>0$,
there exist
$\{ \lambda_j \}^N_{j=1} \subset {\mathbf C}$ and $\{ \eta_j \}^N_{j=1} \subset
{\mathbf R}^n$
such that
$$
\Big\| \widehat{g}   - \sum^N_{j=1} \lambda_j \widehat{f} (\xi - \eta_j)
  \Big\|_{W({\mathcal F}L^p, \ell^1_s)}
= \Big\| g - \sum^{N}_{j=1} \lambda_j e^{i n_j x} f(x) \Big\|_{M^{p,1}_s} < \varepsilon
$$
holds.
Therefore, ${\mathcal F}(A)$ is dense in $W({\mathcal F}L^p, \ell^1_s) ({\mathbf R}^n)$.
Conversely, we  assume that ${\mathcal F}(A)$ is dense in $W({\mathcal F}L^p, \ell^1_s) ({\mathbf R}^n)$.
Then the set  $A$ in Remark \ref{modulation invariance of I}
is dense in $M^{p,1}_s({\mathbf R}^n)$.
Hence Theorem \ref{ideal generated by f} implies that
 $f(x) \not= 0$ for all $x \in {\mathbf R}^n$.
\end{proof}
\begin{remark}
The same principles apply for the case of e.g.
finitely many functions
$f_1, \cdots, f_\ell$
without common zero.
\end{remark}


\section{Set of spectral synthesis
for $M_s^{p,1}({\bf R}^n)$}
\label{Set of spectral synthesis in Mp1s}

In this section, we consider the  set of spectral synthesis
for $M_s^{p,1}({\bf R}^n)$.
Throughout this section, $X$ stands for $M^{p,1}_s({\mathbf R}^n)$
($1 \leq p< \infty$, $s \geq 0$)
or ${\mathcal F}L^1_s ({\mathbf R}^n)$
$(s \geq 0)$.


\begin{definition}
Let $E$ be a closed subset of ${\mathbf R}^n$
and $I(E)$ be the set defined in Lemma
\ref{Basic lemma 1 closed ideal}.
Define $J (E)$ by the closed ideal in $X$
generated by
$$
J_0 (E) =\{ f \in X ~|~ f (x)= 0 ~{\rm in ~ a~neighborhood~ of}~E \}.
$$
Then $E$ is called a {\it set of spectral synthesis} for $X$
if  $I(E) = J(E)$.
\end{definition}

\begin{remark}
\label{translate of spectral synthesis}
We can easily see  that $E$ is a set of spectral synthesis
for $X$ if and only if
$x+ E$ is  a set of spectral synthesis
for $X$ for some/any $x \in {\mathbf R}^n$.
\end{remark}

In the following, we shall prove the following.
\begin{theorem}
\label{equality of the set of spectral synthesis}
Let $1 \leq p \leq 2 $, $s \geq 0$ and $K$ be a compact subset of ${\mathbf R}^n$.
Then $K$ is a set of spectral synthesis for $M^{p,1}_s({\mathbf R}^n)$ if and only
if $K$ is a set of spectral synthesis for ${\mathcal F}L^1_s({\mathbf R}^n)$.
\end{theorem}

\subsection{Technical lemmas}

To prove Theorem \ref{equality of the set of spectral synthesis},
we  use
the ``ideal theory for  Segal algebras''
developed in
Reiter \cite[Ch.6, \S 2]{Reiter}
(see also \cite{Burnham}, \cite{Reiter-Stegeman 2000}).
In the following, we denote by $({\mathcal F}L^1_s)_c$ the space defined in Lemma \ref{inclusion L_0}.

\begin{lemma}
\label{characterization of closed ideal in FL1s}
Let $s \geq 0$ and $I$ be a closed ideal in ${\mathcal F}L^1_s ({\mathbf R}^n)$.
Then $I$ is the closure of $I \cap ({\mathcal F}L^1_s)_c$ in ${\mathcal F}L^1_s({\mathbf R}^n)$.
\end{lemma}

\begin{proof}
Since $I \cap ({\mathcal F}L^1_s)_c \subset I$ and $I$ is closed in ${\mathcal F}L^1_s ({\mathbf R}^n)$,
 the closure of $I \cap ({\mathcal F}L^1_s)_c$ in ${\mathcal F}L^1_s ({\mathbf R}^n)$ is contained in $I$.
We next show that $I$ is contained in the
closure of $I \cap ({\mathcal F}L^1_s)_c$ in ${\mathcal F}L^1_s ({\mathbf R}^n)$.
Let $f \in I$.
Then we have  by Lemma \ref{approximate units in FL1s}
that for any $\varepsilon >0$, there exists $\phi \in C^\infty_c ({\mathbf R}^n)$
such that $ \| f -  \phi f \|_{{\mathcal F}L^1_s} < \varepsilon $ holds.
Since $I$ is an ideal in ${\mathcal F}L^1_s ({\mathbf R}^n)$, we have
$\phi f \in I \cap ({\mathcal F}L^1_s)_c$.
Hence we obtain the desired result.
\end{proof}

\begin{lemma}
\label{closed ideal in Mp1s}
Let $1\leq p \leq 2 $ and $s \geq 0$.
Suppose that $I$ and $I^\prime$ are closed ideals in
${\mathcal F}L^1_s({\mathbf R}^n)$.
Then we have the following:
\begin{itemize}

\item[$(i)$] $I \cap M^{p,1}_s ({\mathbf R}^n)$ is a
closed ideal in $M^{p,1}_s({\mathbf R}^n)$.

\item[$(ii)$] If $I \cap M^{p,1}_s({\mathbf R}^n)
= I^\prime \cap M^{p,1}_s({\mathbf R}^n)$,
then we have $I=I^\prime$.
\end{itemize}
\end{lemma}

\begin{proof}
$(i)$ We first prove that $I \cap M^{p,1}_s({\mathbf R}^n)$ is an
ideal in $M^{p,1}_s({\mathbf R}^n)$.
Let $f \in I \cap M^{p,1}_s({\mathbf R}^n)$ and $g \in M^{p,1}_s({\mathbf R}^n)$.
Since $M^{p,1}_s({\mathbf R}^n)$
is a multiplication algebra, we see $fg \in M^{p,1}_s({\mathbf R}^n)$.
Moreover, since $M^{p,1}_s({\mathbf R}^n) \hookrightarrow {\mathcal F}L^1_s({\mathbf R}^n)$
and $I$ is an ideal in ${\mathcal F}L^1_s({\mathbf R}^n)$,
we have $fg \in I$.
Hence we $fg \in I \cap M^{p,1}_s({\mathbf R}^n)$
for any $f \in I \cap M^{p,1}_s({\mathbf R}^n)$ and
$g \in M^{p,1}_s({\mathbf R}^n)$,
which implies the desired result.
We next prove $I \cap M^{p,1}_s({\mathbf R}^n)$ is closed
in $M^{p,1}_s({\mathbf R}^n)$.
Let $f$ be in the closure of $I \cap M^{p,1}_s({\mathbf R}^n)$ in
$M^{p,1}_s({\mathbf R}^n)$.
Then, there exists $\{ f_n \}_{n=1}^\infty \subset I \cap M^{p,1}_s({\mathbf R}^n)$ such that
$\| f_n - f \|_{M^{p,1}_s} \to 0$ $(n \to \infty)$.
We note that
$M^{p,1}_s ({\mathbf R}^n) \hookrightarrow {\mathcal F}L^1_s({\mathbf R}^n)$,
and thus
$\| f_n - f \|_{{\mathcal F}L^1_s} \to 0$ $(n \to \infty)$.
Since
$I$ is closed in ${\mathcal F}L^1_s({\mathbf R}^n)$
and $M^{p,1}_s({\mathbf R}^n)$ is complete,
we obtain $f \in I \cap M^{p,1}_s({\mathbf R}^n) $,
which implies the desired result.

\noindent
$(ii)$
Since
$
I \cap M^{p,1}_s  \cap ({\mathcal F}L^1_s)_c
= I^\prime \cap M^{p,1}_s  \cap ({\mathcal F}L^1_s)_c
$
and $({\mathcal F}L^1_s)_c \hookrightarrow M^{p,1}_s({\mathbf R}^n)$,
we have $I \cap ({\mathcal F}L^1_s)_c = I^\prime \cap ({\mathcal F}L^1_s)_c$.
Thus it follows from Lemma \ref{characterization of closed ideal in FL1s}
that $I= I^\prime$, which implies the desired result.

\end{proof}

\begin{proposition}
\label{closed ideal in Mp1s 2}
Let $1 \leq p \leq 2$ and $s \geq 0$.
Then for any closed ideal $I_M$ in $M^{p,1}_s({\mathbf R}^n)$,
the ideal $I_F$ in
${\mathcal F}L^1_s({\mathbf R}^n)$
defined by the closure of $I_M$ in ${\mathcal F}L^1_s({\mathbf R}^n)$
satisfies
$I_M= I_F \cap M^{p,1}_s({\mathbf R}^n)$.

\end{proposition}

\begin{proof}
Let $I_M$ be a closed ideal in $M^{p,1}_s ({\mathbf R}^n)$
and define  $I_F^\prime$ by the closure of
$I_M \cap ({\mathcal F}L^1_s)_c$
in ${\mathcal F}L^1_s({\mathbf R}^n)$.
Then $I_F^\prime$ is a closed ideal in ${\mathcal F}L^1_s({\mathbf R}^n)$.
To see $I_F^\prime$ is an ideal in  ${\mathcal F}L^1_s({\mathbf R}^n)$,
let $f \in I_F^\prime$ and $g \in {\mathcal F}L^1_s({\mathbf R}^n)$.
Then there exists $\{ f_n \}^\infty_{n=1} \subset I_M \cap ({\mathcal F}L^1_s)_c $
such that
$\| f - f_n \|_{{\mathcal F}L^1_s} \to 0$ $(n \to \infty)$.
Since $f_n \in  ({\mathcal F}L^1_s)_c$, there exists $\psi_n \in C^\infty_c({\mathbf R}^n)$
such that $\psi_n (x)=1$ on ${\rm supp~}f_n$.
Then we have $\psi_n g \in ({\mathcal F}L^1_s)_c \hookrightarrow M^{p,1}_s({\mathbf R}^n)$.
Therefore,
$\psi_n g \cdot f_n \in I_M$,
and thus
$f_n g =  f_n  \cdot  \psi_n g \in I_M \cap ({\mathcal F}L^1_s)_c$.
Moreover, since
\begin{align*}
\| f g  -  f_n g \|_{{\mathcal F}L^1_s}
\lesssim
\| f -f_n \|_{{\mathcal F}L^1_s}  \| g \|_{{\mathcal F}L^1_s}
\to 0 \quad (n \to \infty),
\end{align*}
we have $fg \in I_F^\prime$.
Hence,  $I_F^\prime$ is an ideal in  ${\mathcal F}L^1_s({\mathbf R}^n)$.

Next, we prove that
$I_F$
is equal to $I_F^\prime$.
Let $f \in I_F$.
Then for any $\varepsilon >0$,
there exists $g \in  I_M$ such that
$\| f - g \|_{{\mathcal F}L^1_s} < \varepsilon$.
On the other hand, Theorem \ref{f-psif} implies that
there exists $\phi \in C^\infty_c({\mathbf R}^n)$
such that
$\| g - \phi g \|_{M^{p,1}_s} < \varepsilon$.
Hence we have $\phi g \in I_M \subset M^{p,1}_s({\mathbf R}^n) \hookrightarrow
{\mathcal F}L^1_s({\mathbf R}^n)$.
Hence we obtain $\phi g \in I_M \cap ({\mathcal F}L^1_s)_c$
and
\begin{align*}
\| f -\phi g \|_{{\mathcal F}L^1_s}
& \leq
\| f - g \|_{{\mathcal F}L^1_s}
+ \| \phi g - g \|_{{\mathcal F}L^1_s} \\
&\lesssim
\| f - g \|_{{\mathcal F}L^1_s}
+ \| \phi g - g \|_{ M^{p,1}_s}
\lesssim \varepsilon ,
\end{align*}
which yields $I_F \subset I_F^\prime$.
The reverse inclusion is clear.

Finally, we prove $I_M = I_F \cap M^{p,1}_s({\mathbf R}^n)$.
Since $I_M \subset I_F$ and $I_M \subset M^{p,1}_s({\mathbf R}^n)$,
we see that
$I_M = I_M \cap M^{p,1}_s({\mathbf R}^n) \subset I_F \cap M^{p,1}_s({\mathbf R}^n)$.
On the other hand, let $f \in  I_F \cap M^{p,1}_s({\mathbf R}^n) $ and $\varepsilon>0$.
By Theorem \ref{f-psif}, there exists $\phi \in C^\infty_c({\mathbf R}^n)$
such that
$\| f - \phi f \|_{M^{p,1}_s} < \varepsilon$.
Now we take
$\varphi \in C^\infty_c({\mathbf R}^n) $ with
$\varphi(x) =1$ on ${\rm supp~} \phi$.
Since $f \in I_F= I_F^\prime$,
there exists $h \in I_M \cap ({\mathcal F}L^1_s)_c$
such that
$\| f -h \|_{{\mathcal F}L^1_s} < \frac{\varepsilon  }{ \| \varphi \|_{M^{p,1}_s} \| \phi \|_{{\mathcal F}L^1_s}  }$.
Then we have $\phi h \in I_M$.
Since ${\mathcal F}L^1_s ({\mathbf R}^n)  \hookrightarrow M^{\infty,1}_s({\mathbf R}^n) $
(see the proof of Lemma \ref{inclusion L_0}),
we obtain
\begin{align*}
\| f - \phi h \|_{M^{p,1}_s}
& \leq
\| f - \phi f \|_{M^{p,1}_s} + \| \varphi \phi (f-h) \|_{M^{p,1}_s} \\
&\lesssim \| f - \phi f \|_{M^{p,1}_s} + \|  \varphi \|_{M^{p,1}_s}
\|\phi (f-h) \|_{M^{\infty,1}_s}  \\
&\lesssim \| f - \phi f \|_{M^{p,1}_s}+ \|  \varphi \|_{M^{p,1}_s}
\|\phi (f-h) \|_{{\mathcal F}L^1_s} \\
& \lesssim
\| f - \phi f \|_{M^{p,1}_s} + \| \varphi \|_{M^{p,1}_s} \| \phi \|_{{\mathcal F}L^1_s}
\|  f-h \|_{{\mathcal F}L^1_s}.
\end{align*}
Therefore $f$ is in the closure of $I_M$ in $M^{p,1}_s({\mathbf R}^n)$.
Since $I_M$ is closed in $M^{p,1}_s({\mathbf R}^n)$,
we get the desired result.

\end{proof}

\begin{remark}
Let $I_M$ and $I_M^\prime$ be closed ideals in
$M^{p,1}_s({\mathbf R}^n)$, and
$I_F$ be the closure of $I_M$ in ${\mathcal F}L^1_s ({\mathbf R}^n)$.
If  the closure of $I_M^\prime$  in ${\mathcal F}L^1_s ({\mathbf R}^n)$
is equal to $I_F$, then
Proposition \ref{closed ideal in Mp1s 2} implies that $I_M  = I_M^\prime$.

\end{remark}

Combining Lemma \ref{closed ideal in Mp1s}
and Proposition \ref{closed ideal in Mp1s 2},
we obtain the following result.

\begin{theorem}[The ideal theory for  Segal algebras]
Let $1 \leq p \leq 2$, $s \geq 0$,
${\mathcal I}_F$ be the set of all closed ideals in
${\mathcal F}L^1_s({\mathbf R}^n)$,
and ${\mathcal I}_M$ be the set of all closed ideals in
$M^{p,1}_s({\mathbf R}^n)$.
Define the map $\iota :  {\mathcal I}_F \to {\mathcal I}_M $ by
$\iota(I_F) =  I_F \cap M^{p,1}_s({\mathbf R}^n)$.
Then
$\iota$ is bijective.
More precisely, we have
$$
\iota^{-1} (I_M)
= \overline{I_M}^{\| \cdot  \|_{ {\mathcal F}L^1_s }}
\qquad
(I_M \in {\mathcal I}_M)
$$
and
$$
\iota \big(  \overline{I_M}^{\| \cdot  \|_{ {\mathcal F}L^1_s }} \big) =
\overline{I_M}^{\| \cdot  \|_{ {\mathcal F}L^1_s }} \cap M^{p,1}_s({\mathbf R}^n)
= I_M,
$$
where $\overline{I_M}^{\| \cdot  \|_{ {\mathcal F}L^1_s }}$
denotes the closure of $I_M$ in
${\mathcal F}L^1_s({\mathbf R}^n)$.

\end{theorem}

\subsection{The proof of Theorem \ref{equality of the set of spectral synthesis}}

For a closed subset $K$ of  ${\mathbf R}^n$,
we set $I_F(K)=  \{ f \in {\mathcal F}L^1_s ({\mathbf R}^n) ~|~ f|_K =0 \}$
and $I_M(K)=  \{ f \in M^{p,1}_s ({\mathbf R}^n) ~|~ f|_K =0 \}$.
Moreover,  we define $J_F(K)$ by the
closure
of
$
 \{ f \in {\mathcal F}L^1_s ({\mathbf R}^n) ~|~ f (x) =0 {\rm ~in~ a~neighborhood~of~} K  \}
$
 in ${\mathcal F}L^1_s({\mathbf R}^n)$,
and
$J_M(K)$ by the
closure
of
$
 \{ f \in M^{p,1}_s ({\mathbf R}^n) ~|~ f(x) =0 {\rm ~in~ a~neighborhood~of~} K  \}$
 in $M^{p,1}_s({\mathbf R}^n)$.
We note that
 $J_F(K)$ is the smallest closed ideal $I$ in ${\mathcal F}L^1_s ({\mathbf R}^n)$
such that $\bigcap_{f \in I} f^{-1}(\{0 \}) =K$
(cf. \cite[Proposition 2.4.5]{Reiter-Stegeman 2000}).
Similarly,
$J_M(K)$ is the smallest closed ideal $I$ in $M^{p,1}_s  ({\mathbf R}^n)$
such that $\bigcap_{f \in I} f^{-1}(\{0 \}) =K$
(cf. Theorems \ref{image of analytic function} and \ref{image of analytic function 2}).
Thus Proposition \ref{closed ideal in Mp1s 2} implies
that  $I_F(K) = J_M(K)$ if and only if
$I_F(K) = J_F(K) $.
Therefore, $K$ is a set of spectral synthesis for $M^{p,1}_s({\mathbf R}^n)$
if and only if $K$ is a set of spectral synthesis for ${\mathcal F}L^1_s ({\mathbf R}^n)$.

\subsection{Examples}
\label{example of set of spectral synthesis}

As an application of Theorem \ref{equality of the set of spectral synthesis},
we show some  concrete examples of sets of spectral synthesis for
$M^{p,1}_s({\mathbf R}^n)$.

\begin{example}[cf. {\cite[Theorem 2.7.6]{Reiter-Stegeman 2000}}]
Let $1 \leq p \leq 2$.
Then a circle in ${\mathbf R}^2$ is
a  set of spectral synthesis for $M^{p,1}_s({\mathbf R}^2)$
if $0 \leq s < \frac{1}{2}$, but
not if $s \geq \frac{1}{2}$.

\end{example}

\begin{example}[cf. {\cite[Theorem 2.7.7]{Reiter-Stegeman 2000}}]
Let $1 \leq p \leq 2$
and $s \geq 0$. Then  the sphere $S^{n-1} \subset {\mathbf R}^n$
is not a set of spectral synthesis for
$M^{p,1}_s({\mathbf R}^n)$ if $n \geq 3$.

\end{example}

\begin{example}[cf. {\cite[Theorem 2.7.9]{Reiter-Stegeman 2000}}]
Let $1 \leq p \leq 2$.
Single points of ${\mathbf R}^n$ are sets of spectral synthesis
for $M^{p,1}_s({\mathbf R}^n)$, if $0 \leq s <1$.

\end{example}

\begin{example}[cf. {\cite[Theorem 2.7.10]{Reiter-Stegeman 2000}}]
Let $1 \leq p \leq 2$
and $s \geq 0$.
Then a closed ball in ${\mathbf R}^n$ is a set of spectral
synthesis for $M^{p,1}_s({\mathbf R}^n)$
and so is the complement of an open ball in
${\mathbf R}^n$.

\end{example}

\section{Spectral synthesis revisited}
\label{Appendix}

As mentioned  in Section \ref{example of set of spectral synthesis},
if $1 \leq p \leq 2$ and $0 \leq s <1$, then
single points of ${\mathbf R}^n$ are sets of spectral synthesis
for $M^{p,1}_s({\mathbf R}^n)$.
In this section we will prove it directly without using Theorem
\ref{equality of the set of spectral synthesis},
and this is also true for $p >2$.

\begin{theorem}
\label{1 point spectral}
Let $1 \leq p < \infty$, $0 \leq s <1$
and $x_0 \in {\mathbf R}^n$.
Then $\{ x_0 \}$ is a set of spectral synthesis
in $M^{p,1}_s({\mathbf R}^n)$.

\end{theorem}

\subsection{Technical Lemma}

\begin{lemma}
\label{convolution decomposition}

For any $t_0  \in{\bf R}^n$ and $ R > 0 $,
there exist $\psi^{(1)} ,  \psi^{(2)} \in C_c^\infty ({\mathbf R}^n)$ such that
$\psi= \psi^{(1)} *  \psi^{(2)}$
satisfies
$\psi (x)=1$ on $B_R(t_0)$
and
${\rm supp~}  \psi \subset B_{5R}(t_0)$.

\end{lemma}

\begin{proof}
Let $g \in C_c^\infty({\bf R}^n)$ be  such that
$g(x) \geq 0$ $(x \in {\mathbf R}^n)$,
$g(x) = 1$ on $B_R(0)$
and
${\rm supp~} g \subset  B_{2R}(0)$.
Define $\psi^{(1)}, \psi^{(2)} \in C^\infty_c({\mathbf R}^n)$ by
$$
\psi^{(1)} (x)= g \Big( \frac{x}{2} \Big)
\quad
{\rm and}
\quad
\psi^{(2)} (x) = \frac{2^n}{ \| g  \|_{L^1}} g(2(x-t_0)),
$$
respectively, and set $\psi= \psi^{(1)} * \psi^{(2)}$.
Then $\psi$ satisfies  the desired conditions.
Indeed, we note that
\begin{align*}
\psi(x) &=
\int_{ {\mathbf R}^n }  \psi^{(1)}  (x-y)  \psi^{(2)} (y) dy
=
 \frac{2^n}{  \| g  \|_{L^1}}
\int_{ {\mathbf R}^n }
g \Big( \frac{x- y}{2} \Big)g(2(y -t_0) )dy.
\end{align*}
Moreover,   if
$x \in B_R(t_0)$ and
$|2(y-t_0) | \leq 2 R$, then
we have
$
\frac{|x-y|}{2} \leq
\frac{|x-t_0|}{2} + \frac{ |t_0-y| }{2} \leq R,
$
and thus $g ( \frac{x-y}{2} ) =1$.
Therefore, we obtain
\begin{align*}
\psi(x) & =  ( \psi^{(1)} * \psi^{(2)})  (x)
= \frac{2^n}{  \|  g  \|_1}\int_{{\mathbf R}^n}  g(2(y -t_0) ) dy = 1
\end{align*}
on $B_R(t_0)$.
Also we have
$
{\rm supp~} \psi^{(1)}  \subset \{ x \in {\mathbf R}^n  ~|~ |x| \leq 4 R \}
$
and
$
{\rm supp~ } \psi^{(2)} \subset \{ x \in {\mathbf R}^n ~|~ |x- t_0| \leq R \},
$
and thus we see
$
{\rm supp~} \psi
\subset
{\rm supp ~} \psi^{(1)} +
{\rm supp~} \psi^{(2)}
\subset  \{x  \in {\mathbf R}^n ~|~  |x- t_0|\leq 5 R \} .
$

\end{proof}

\begin{lemma}
\label{estimates of FT of psi}
Let
$\psi^{(1)}, \psi^{(2)} \in C^\infty_c ({\mathbf R}^n)$
be such that
${\rm supp~} \psi^{(j)} \subset B_R(0)$
for some $R>0$ $(j=1,2)$.
Define $\psi = \psi^{(1)} * \psi^{(2)}$.
Then we have for $0 \leq s <1$ and $ \vartheta \in {\mathbf R}^n$
$$
\int_{{\mathbf R}^n}
\langle \xi \rangle^s
\big| \widehat{\psi}
\left(
\xi- \vartheta
\right) - \widehat{\psi}(\xi) \big|
d\xi
\leq C_\psi
|  \vartheta  |^s
\Big(
\max_{|t|\leq R}|e^{i  \vartheta t  }-  1|
\Big)^{1-s}.
$$

\end{lemma}

\begin{proof}
We first  note that
since $\widehat{\psi} = \widehat{\psi^{(1)}} \cdot \widehat{\psi^{(2)}}$, we have
\begin{align*}
&
\widehat{\psi} ( \xi-  \vartheta) - \widehat{\psi}(\xi)\\
&=
\big( \widehat{\psi^{(1)}} (\xi-  \vartheta )-
\widehat{\psi^{(1)}}(\xi) \big)
\widehat{\psi^{(2)}} ( \xi -   \vartheta )
+\widehat{\psi^{(1)}}(\xi)
\big( \widehat{\psi^{(2)}}
(\xi-    \vartheta  ) - \widehat{\psi^{(2)}}(\xi)  \big).
\end{align*}
Then it follows from
 the Cauchy-Schwarz inequality
and the Plancherel theorem
that
\begin{align*}
&
\int_{{\mathbf R}^n}
\langle \xi \rangle^s
\big| \widehat{\psi}  (\xi-   \vartheta   )
- \widehat{\psi}(\xi) \big|
d\xi\\
&
\leq
\Big(
\int_{{\mathbf R}^n}
\big|  \widehat{\psi^{(1)}}
( \xi-   \vartheta  ) -
\widehat{\psi^{(1)}}(\xi) \big|^2d\xi \big)^{\frac{1}{2}}
\Big( \int_{{\mathbf R}^n}
\langle \xi \rangle^{2s}
\big| \widehat{\psi^{(2)}}
( \xi-   \vartheta  ) \big|^2
d \xi
\Big)^{\frac{1}{2}}\\
& \quad +
\Big(
\int_{{\mathbf R}^n}
\big|  \widehat{\psi^{(2)}}
( \xi-   \vartheta  ) -
\widehat{\psi^{(2)}}(\xi) \big|^2d\xi \big)^{\frac{1}{2}}
\big( \int_{{\mathbf R}^n}
\langle \xi  \rangle^{2s}
\big| \widehat{\psi^{(1)}}
( \xi-   \vartheta   )  \big|^2
d\xi \Big)^{\frac{1}{2}}\\
&=
\| {\mathcal F}_{x \to \xi}
[ (e^{ i  \vartheta x}-1)\psi^{(1)} (x) ](\xi)
\|_{L^2({\mathbf R}^n_\xi)}
\langle
 \vartheta
\rangle^s
\|  \langle \xi \rangle^s
\widehat{\psi^{(2)}}  (\xi) \|_{L^2({\mathbf R}^n_\xi)} \\
& \quad +
\| {\mathcal F}_{x \to \xi}
[ (e^{ i  \vartheta   x}-1)\psi^{(2)} (x) ] (\xi)
 \|_{L^2({\mathbf R}^n_\xi)}
\langle
 \vartheta
\rangle^s
\|  \langle \xi \rangle^s
\widehat{\psi^{(1)}}  (\xi) \|_{L^2({\mathbf R}^n_\xi)}.
\end{align*}
We note that
$|e^{ i  \vartheta x  } -1 |
\leq \min \{ 2,    |   \vartheta  x  | \}$,
and
for $j=1,2$,
it follows from the Plancherel Theorem that
$$
\| {\mathcal F}_{x \to \xi}
[ (e^{ i  \vartheta x}-1)\psi^{(j)} (x) ] (\xi)
 \|_{L^2({\mathbf R}^n_\xi)}
=
(2 \pi)^{ \frac{n}{2} }
\|
 (e^{ i  \vartheta x}-1)\psi^{(j)} (x)  \|_{L^2({\mathbf R}^n_x)}.
$$
Moreover, since ${\rm supp~} \psi^{(j)}  \subset B_R(0)$ $(j=1,2)$,
we obtain
\begin{align*}
&
\| {\mathcal F}_{x \to \xi}
[ (e^{ i  \vartheta x}-1)\psi^{(j)} (x) ]  (\xi)
\|_{L^2({\mathbf R}^n_\xi)}
(
1 + |   \vartheta |^s )
\\
&\lesssim
\Big( \max_{ |x| \leq R}|e^{i
 \vartheta  x}-1| \Big)^{1-s}
\Big( \int_{{\mathbf R}^n}  \big( |e^{i   \vartheta  x}-1|^s|
\psi^{(j)}  (x)|  \big)^2dx \Big)^{\frac{1}{2}}
(1 + |   \vartheta |^s )
\\
&\leq
|   \vartheta |^s
\Big( \max_{ |x| \leq R}|e^{i   \vartheta  x}-1| \Big)^{1-s}
\Big( \int_{{\mathbf R}^n}  \big(  |x|^s|
\psi^{(j)} (x)|  \big)^2dx \Big)^{\frac{1}{2}}  \\
& \quad +
2^s |  \vartheta  |^s
\Big( \max_{|x|\leq R}|e^{i  \vartheta  x}  -1| \Big)^{1-s}
\Big( \int_{{\mathbf R}^n}
| \psi^{(j)}  (x)|^2dx \Big)^{\frac{1}{2}},
\end{align*}
which yields the desired inequality.

\end{proof}

Now we prepare a lemma,
which corresponds to a weighted version of
Bhimani-Ratnakumar
\cite[Proposition 3.14]{Bhimani Ratnakumar}.

\begin{lemma}
\label{weighted Bhimani Ratnakumar lemma}

Let  $0 \leq s < 1$, $f \in M^{1,1}_s ({\mathbf R}^n)$, $x_0 \in {\mathbf R}^n$
and $\varepsilon >0$.
Then there exists $\phi \in C_c^\infty ({\mathbf R}^n)$ such that
\begin{itemize}
\item[$(i)$] $\|  (f- f (x_0) ) \phi \|_{M^{1,1}_s} < \varepsilon$
\item[$(ii)$]
$\phi(x) =1$ in some neighborhood of $x_0$.
\end{itemize}
\end{lemma}

\begin{proof}

Take   $\psi = \psi^{(1)} * \psi^{(2)}  \in C^\infty_c ({\mathbf R}^n)$
as in Lemma \ref{convolution decomposition}
with $t_0=0$ and $R=1$,
i.e.,
$\psi (x) =1$ on $B_1(0)$
and
${\rm supp~} \psi \subset  B_5(0)$.
Define $\psi_\lambda(x)=\psi(\lambda x)$ and
$
 h^\lambda(x)= (f(x)-f(x_0)) \psi_\lambda(x-x_0)
$
for $\lambda >0$.
We note that if
$\lambda >5$
and $x \in {\rm supp~} h^\lambda$,
then
$\psi (x - x_0) =1$.
Thus
$h^\lambda(x)=h^\lambda(x)\psi(x-x_0)$.

Without loss of generality, we may assume   $x_0=0$
(see Remark \ref{translate of spectral synthesis}).
Let $g_0 (t) = e^{- \frac{|t|^2}{2}}$ $(t \in {\mathbf R}^n)$
and $\lambda >5$.
By \eqref{basic STFT} in Section \ref{STFT},
we have
\begin{align*}
\| h^\lambda \|_{M^{1,1}_s}
&= \| \|  \langle \xi \rangle^s  V_{g_0} h^\lambda
 (x,\xi) \|_{L^1({\mathbf R}^n_x)}
\|_{L^1 ({\mathbf R}^n_\xi)}  \\
&=
(2 \pi)^{-n}
\| \| \langle \xi \rangle^s
V_{ \widehat{g_0} } \widehat{ h^\lambda}  (\xi, -x) \|_{L^1({\mathbf R}^n_x)}
\|_{L^1({\mathbf R}^n_\xi)}.
\end{align*}
Since  $\widehat{g_0} =(2 \pi)^{ \frac{n}{2} }  g_0$, $g_0^* =g_0$ and
$h^\lambda (x) =h^\lambda (x) \psi (x)$,
we obtain by \eqref{basic STFT} in Section \ref{STFT}
\begin{align*}
V_{ \widehat{g_0}}   \widehat{ h^\lambda } (\xi, -x)
&= (2 \pi)^{ - \frac{n}{2} }
V_{g_0}  ( \widehat{h^\lambda} * \widehat{\psi}) (\xi, -x)
= (2 \pi)^{ - \frac{n}{2} }
e^{ i x \xi}(\widehat{h^\lambda} * \widehat{\psi} * M_{-x}g_0)(\xi).
\end{align*}
Thus by the Fubini Theorem and the Minkowski inequality for integral that
\begin{align*}
\| h^\lambda \|_{M^{1,1}_s}
&
=
(2 \pi)^{-n}
\| \| \langle \xi \rangle^s
V_{ \widehat{g_0} } \widehat{  h^\lambda  }  (\xi, -x) \|_{L^1({\mathbf R}^n_x)}
\|_{L^1({\mathbf R}^n_\xi)} \\
&\approx
\| \| \langle \xi \rangle^s
(\widehat{h^\lambda} * \widehat{\psi} * M_{-x}g_0)(\xi)
 \|_{L^1({\mathbf R}^n_\xi)}
\|_{L^1({\mathbf R}^n_x)} \\
& \leq
\|
\| \langle \xi \rangle^s \widehat{ h^\lambda  } (\xi)
 \|_{L^1({\mathbf R}^n_\xi)}
\| \langle \xi \rangle^s (\widehat{\psi} *M_{-x} g_0) (\xi)
\|_{L^1 ({\mathbf R}^n_\xi)}
\|_{L^1({\mathbf R}^n_x)} \\
&=
\| \langle \cdot \rangle^s \widehat{ h^\lambda  }
 \|_{L^1}
\| \psi \|_{M^{1,1}_s}.
\end{align*}
Since
$
\widehat{ h^\lambda  } (\zeta)
= (2 \pi)^{-n} (\widehat{ \psi_\lambda } * \widehat{f}) (\zeta)
- f(0) \widehat{ \psi_\lambda } (\zeta)
$
and
$\widehat{ \psi_\lambda  } (\zeta)
= \frac{1}{ \lambda^n } \widehat{\psi}  ( \frac{\zeta}{\lambda} )  $,
we have
\begin{align*}
\widehat{h^\lambda}(\zeta)
&=
\frac{1}{(2 \pi)^n}
\int_{{\mathbf R}^n}
\widehat{f}( \eta) \widehat{\psi_\lambda}
(\zeta- \eta) d \eta -
\frac{1}{(2 \pi)^n}
\Big(
\int_{{\mathbf R}^n}
\widehat{f} (\eta) d \eta
\Big)
\widehat{\psi_\lambda}(\zeta) \\
&=
\frac{1}{ (2 \pi \lambda)^n}
\int_{{\mathbf R}^n}
\widehat{f} (\eta)
\Big(
\widehat{\psi}
\Big( \frac{\zeta - \eta}  {\lambda}
\Big) - \widehat{\psi} \Big( \frac{\zeta}{\lambda} \Big)
\Big) d  \eta
\end{align*}
and thus we obtain by Lemma \ref{estimates of FT of psi}
choosing $\vartheta = \frac{\xi}{\lambda}$
and $\lambda >5$
that
\begin{align*}
\| \langle \cdot \rangle^s \widehat{ h^\lambda  }
 \|_{L^1({\mathbf R}^n)}
&\lesssim
\int_{{\mathbf R}^n}
\Big(
\frac{1}{\lambda^n}
\int_{{\mathbf R}^{n}}
 \langle  \zeta  \rangle^s
\Big| \widehat{\psi}  \Big(  \frac{\zeta - \eta}
{\lambda}  \Big) - \widehat{\psi} \Big( \frac{\zeta}{\lambda} \Big) \Big|
 d \zeta \Big) |\widehat{f}(\eta)|
 d \eta  \\
&
\lesssim
\lambda^s
\int_{{\mathbf R}^n}
\Big(
\int_{{\mathbf R}^n}
\langle \xi \rangle^s
\Big|\widehat{\psi}
\Big( \xi - \frac{\eta}{\lambda} \Big) - \widehat{\psi}(\xi) \Big|
 d \xi
\Big) |\widehat{f} (\eta)| d \eta \\
& \lesssim
\int_{{\mathbf R}^n}
| \eta |^s
\left(
\max_{|t|\leq 4}|e^{i\frac{\eta}{\lambda}t}-1|
\right)^{1-s}
|\widehat{f} (\eta)|
d \eta.
\end{align*}
We note that  $M^{1,1}_s ({\mathbf R}^n) \hookrightarrow
{\mathcal F}L^1_s ({\mathbf R}^n)$,
and  also
$\left(
\max_{|t|\leq 4}|e^{i\frac{\eta}{\lambda}t}-1|
\right)^{1-s} \leq 2$ and
$
\left(
\max_{|t|\leq 4}|e^{i\frac{\eta}{\lambda}t}-1|
\right)^{1-s} \to 0
$
as $\lambda \to \infty$.
Therefore the Lebesgue convergence theorem yields
that
$\| \langle \cdot \rangle^s \widehat{ h^\lambda  }
 \|_{L^1({\mathbf R}^n)}
\to 0 $
$(\lambda \to \infty)$.
Therefore,
for any $\varepsilon >0$, there exists
$\lambda_0 >0$ such that
$
\| h^{\lambda_0} \|_{M^{1,1}_s} <   \varepsilon$.
Hence,
by putting  $\phi (x)=\psi_{\lambda_0} (x)$,
we have the desired result.

\end{proof}

Next we  prepare a lemma,
which corresponds to
a weighted version of
Bhimani \cite[Proposition 4.7]{Bhimani}.

\begin{lemma}
\label{weighted Bhimani lemma}
Let  $1 \leq p \leq \infty$, $0 \leq s < 1$, $f \in M^{p,1}_s ({\mathbf R}^n)$, $x_0 \in {\mathbf R}^n$
and $\varepsilon >0$.
Then there exists $\tau \in C_c^\infty ({\mathbf R}^n)$ such that

\begin{itemize}

\item[$(i)$] $\| (f- f (x_0) )  \tau  \|_{M^{p,1}_s} < \varepsilon$

\item[$(ii)$]
$\tau (x)=1$ in some neighborhood of $x_0$.

\end{itemize}

\end{lemma}

\begin{proof}
Let $\psi \in C^\infty_c({\mathbf R}^n)$
be such that $\psi(x) =1$ on some neighborhood of
$x_0$, and set
$h(x) =(f(x) - f(x_0))  \psi (x)   $.
We note that $h(x_0) =0$.
Since $h \in M^{1,1}_s ({\mathbf R}^n)$,
Lemma \ref{weighted Bhimani Ratnakumar lemma}
implies that there exists $\phi \in C^\infty_c ({\mathbf R}^n)$
such that
$
\|(h -h(x_0))  \phi  \|_{M^{1,1}_s} < \varepsilon
$
and $\phi =1$ on some neighborhood of $x_0$.
Now we define
$\tau = \psi \phi \in C^\infty_c ({\mathbf R}^n) $.
Then we see that  $\tau =1$ on some neighborhood of $x_0$,
and we have by $M^{1,1}_s ({\mathbf R}^n) \hookrightarrow  M^{p,1}_s ({\mathbf R}^n)$
\begin{align*}
\| (f- f(x_0)) \tau  \|_{M^{p,1}_s}
=
\|   (h - h(x_0) ) \phi  \|_{M^{p,1}_s}
 \lesssim  \| (h- h(x_0)) \phi  \|_{M^{1,1}_s}.
\end{align*}
Hence we obtain the desired result.

\end{proof}

\begin{remark}
Recall that $M^{p,1}_s ({\mathbf R}^n)$
is a Banach algebra.
Then it follows from Lemma \ref{weighted Bhimani lemma}
that $M^{p,1}_s ({\mathbf R}^n)$
satisfies the condition of Wiener-Ditkin
(cf. \cite[Ch.2, \S 4.3]{Reiter}, \cite[Definition 2.4.7]{Reiter-Stegeman 2000}),
i.e., for every point $x_0 \in {\mathbf R}^n$
the following holds:
for any function $f \in M^{p,1}_s({\mathbf R}^n)$
vanishing at $x_0$ and any neighborhood ${\mathcal U}_0$ of $0$ in
$M^{p,1}_s({\mathbf R}^n)$,
there exists $\tau \in M^{p,1}_s({\mathbf R}^n)$
such that  $(i)$ $\tau$ is constant $1$ near
$x_0$, and $(ii)$ $f \cdot \tau \in {\mathcal U}_0$.

\end{remark}

\subsection{The proof of Theorem \ref{1 point spectral}}
We first note that
$J_0 ( \{ x_0 \} )  \subset I( \{ x_0 \} )$.
Since $J(\{ x_0 \})$ is the smallest closed ideal
containing $J_0(\{ x_0 \})$,
it follows from   Lemma \ref{Basic lemma 1 closed ideal} that
$J(\{ x_0 \} )  \subset I(\{ x_0 \} )$.
We next  prove $I(\{ x_0 \} )  \subset J(\{ x_0 \} )  $.
Since $I(\{ x_0 \})$
is an closed ideal with
$J_0 ( \{ x_0 \} )  \subset I( \{ x_0 \} )$,
it suffices to prove that
$I (\{ x_0 \}) \subset \overline{ J_0 (\{ x_0 \})   }$.
Let $f \in I (\{ x_0 \})$ and $\varepsilon >0$.
We note that $f(x_0)=0$.
By Lemma \ref{weighted Bhimani lemma},
there exists $\tau \in C_c^\infty ({\mathbf R}^n)$ such that
$\| \tau f \|_{M^{p,1}_s} < \varepsilon$ and $\tau (x) =1$
on some neighborhood of $x_0$.
Now we set $g = (1- \tau) f$. Then
we have $g = f - \tau f \in M^{p,1}_s ({\mathbf R}^n)$,
$g (x)=0$ on some neighborhood of $x_0$
and
$\| f -g  \|_{M^{p,1}_s} = \| \tau f \|_{M^{p,1}_s} < \varepsilon$.
Hence we obtain the desired result.


\section{Inclusion relation between $M^{p,1} $ and ${\mathcal F}\hspace{-0.08cm} A_p$}
\label{A relation between Mp10}

In this
section, we consider the inclusion relation between
 the modulation space
$M^{p,1} ({\mathbf R})$ and the Fourier
Segal algebra ${\mathcal F}\hspace{-0.08cm} A_p({\mathbf R})$.
Here, ${\mathcal F}\hspace{-0.08cm} A_p({\mathbf R})$ is the space defined by
the norm
$$
\| f \|_{{\mathcal F}\hspace{-0.08cm} A_p} = \| f \|_{L^p} + \| \widehat{f} \|_{L^1}.
$$
We note that
since
${\mathcal F}\hspace{-0.08cm} A_p({\mathbf R})$ is the Fourier image of
Segal algebra $A_p({\mathbf R})$ which is defined by
the norm
$$
\| f \|_{A_p} = \|  \widehat{f} \|_{L^p} + \| f \|_{L^1}
$$
(see \cite{Reiter}, \cite{Reiter-Stegeman 2000}, \cite{Yap} for more details),
${\mathcal F}\hspace{-0.08cm} A_p({\mathbf R})$ is  a  (abstract) Segal algebra
in $A({\mathbf R}) (= {\mathcal F}L^1_0 ({\mathbf R}))$
in the following sense.

\begin{itemize}
	\item[(i)] ${\mathcal F}\hspace{-0.08cm} A_p({\mathbf R})$ is a commutative Banach algebra with
	pointwise multiplication and the norm
	$\| f \|_{{\mathcal F} \hspace{-0.08cm}  A_p}$.
	
	\item[(ii)] ${\mathcal F}\hspace{-0.08cm} A_p ({\mathbf R})$ is a dense subset
	of  the Fourier algebra $A({\mathbf R})$.
	
	\item[(iii)] ${\mathcal F}\hspace{-0.08cm} A_p({\mathbf R})$ is isometrically invariant under modulation 	operators:
$$\| M_\xi f \|_{{\mathcal F}\hspace{-0.08cm} A_p} = \| f \|_{{\mathcal F}\hspace{-0.08cm} A_p}, \quad \forall \xi \in {\mathbf R}. $$
\end{itemize}
Moreover,  ${\mathcal F}\hspace{-0.08cm} A_p({\mathbf R})$ has similar properties to $A({\mathbf R})$ (see \cite{Lai}, \cite{Larsen}).
On the other hand,
Lemmas \ref{basic pro mod}, \ref{LocFLp}
 and \ref{inclusion Mp1s subset FL1} imply
that
$M^{p,1} ({\mathbf R}) \hookrightarrow {\mathcal F}\hspace{-0.08cm} A_p({\mathbf R}) $ for $1 \leq p \leq 2$.
Therefore, it is natural to ask whether $M^{p,1} ({\mathbf R})$ is a proper subset of ${\mathcal F}\hspace{-0.08cm} A_p({\mathbf R})$.

%
\begin{theorem}
\label{inclusion Mp1 subset Ap}
For $1 \leq p \leq 2$ we have the   proper, dense inclusion
$M^{p,1}({\mathbf R}) \subsetneq {\mathcal F}\hspace{-0.08cm} A_p({\mathbf R})$.
\end{theorem}

We remark that Losert \cite[Theorem 2]{Losert}
implies that  Theorem \ref{inclusion Mp1 subset Ap}
holds for the case  $p=1$.

\subsection{Technical lemma}
To prove Theorem \ref{inclusion Mp1 subset Ap},
we prepare  the following lemma.
\begin{lemma}
\label{existance of discrete measure}
Let $K$ be a compact subset of $ {\mathbf R}$.
Then for any $\varepsilon >0$, there exist a discrete measure $\mu \in M({\mathbf R})$
and $m,N \in {\mathbf N}$ for which the following conditions hold.
\begin{itemize}

\item[$(i)$] $\| \mu \|_{M(\mathbf R{})}=1 $ and $\| \widehat{\mu} \|_{L^\infty}  < \varepsilon$.

\item[$(ii)$]
${\rm supp~} \mu $ consists of $2^m$ different points and is expressed as
$$
{\rm supp~} \mu =
\Big\{  x \in {\mathbf R} ~\Big| ~
x= \sum^{m}_{j=1}  \alpha_j N_j, \  \alpha_j= 0,1, \ N_j = 2^{j-1}N \Big\}.
$$
\item[$(iii)$] The sets
$ \{ - x+ K  ~|~ x \in{\rm supp~ \mu} \}$  are mutually disjoint.
\end{itemize}
\end{lemma}
\begin{proof}
Our proof is based on the argument
due to Kahane  \cite[p.34-36]{Kahane},
which is also
known as the Rudin-Shapiro method.
Let $\varepsilon >0$
and take $m \in {\mathbf N}$
such that  $ 2^{ \frac{1}{2} - \frac{m}{2}  } < \varepsilon$.
We choose $N \in {\mathbf N}$ such that
the sets
$$
\Big\{
- \sum^{m}_{j=1} \alpha_j N_j  + K
~\Big|~ \alpha_j=0,1, \ N_j = 2^{j-1}N
\Big\}
$$
are mutually disjoint for different
choices of $\alpha_j \in \{0,1 \}$,
and then define $ \mu_j,  \nu_j   \in  M({\mathbf R})$
$(j=1,2, \cdots, m)$
by
$\mu_0= \nu_0 = \delta_0$
and
$$
\mu_{j} = \mu_{j-1} + \nu_{j-1} * \delta_{N_{j}},
\quad
\nu_{j} = \mu_{j-1} - \nu_{j-1} * \delta_{N_{j}}
\quad
(j=1, 2, \cdots, m).
$$
We note that the set \{ $\sum^m_{j=1} \alpha_j N_j  ~|~ \alpha_j=0,1 \}$
consists of  $2^m$  different points, and thus we obtain
$\| \mu_m \|_{M({\mathbf R})} = 2^m$.
Moreover, it follows from
$$
(  \mu_{j-1} \pm \mu_{j-1} * \delta_{N_j}  )^\wedge (\xi)
= \widehat{\mu_{j-1}}   (\xi)
\pm
\widehat{ \nu_{j-1}  } (\xi) e^{-i \xi N_{j}}
\quad(j=1, 2, \cdots, m)
$$
that
\begin{align*}
| \widehat{\mu_{m}}  (\xi) |^2 + |  \widehat{\nu_{m}} (\xi)  |^2
&= 2 ( | \widehat{ \mu_{m-1}   } (\xi) |^2 + | \widehat{\nu_{m-1}} (\xi)  |^2 ) \\
&= \cdots  \\
&= 2^m ( | \widehat{\mu_0} (\xi) |^2 + | \widehat{\nu_0} (\xi) |^2 ) \\
&= 2^{m+1}.
\end{align*}
This gives
$$
| \widehat{\mu_{m}} (\xi) |
\leq
(|\widehat{\mu_m} (\xi) |^2 + |\widehat{\nu_m} (\xi) |^2)^{\frac{1}{2}}
\leq 2^{ \frac{m+1}{2}  }.
$$
Now we set
$\mu = 2^{-m} \mu_m $.
Then we can easily see $\| \mu \|_{M({\mathbf R})} =1$
and
$$
\| \widehat{\mu} \|_{L^\infty}
\leq
2^{-m}  2^{ \frac{m+1}{2}   }
= 2^{ \frac{1}{2}  - \frac{m}{2} }    < \varepsilon.
$$
Hence, we obtain the desired result.

\end{proof}

\begin{remark}
\label{express of mu}

We note that $\mu$ is also expressed as
$$
\mu =
2^{-m}
\sum_{ \alpha_j =0,1 }  C_{(\alpha_1, \cdots, \alpha_{2m})} \delta_{\alpha_1 N_1 + \alpha_2 N_2 + \cdots + \alpha_{2m} N_{2m}},
$$
where each term $C_{(\alpha_1, \cdots, \alpha_{2m})} $ is  equal to  $1$ or $-1$.
For the sake of simplicity, we simply write
$\mu = \sum^{2^m}_{j=1} a_j \delta_{\ell(j)}$
with ${\rm supp~} \mu = \{ \ell(j) \in {\mathbf N} ~|~ j=1,2, \cdots, 2^m\}$,
where $a_j = 2^{-m}$ or $a_j= - 2^{-m}$.

\end{remark}

\begin{remark}
\label{modification discrete measure}
Let $1 \leq p < 2$ and  choose $r \in {\mathbf N}$ such that
$2^{ \frac{1}{2} - r ( \frac{1}{p} - \frac{1}{2}  )   } < \varepsilon $.
In the same way as above, we define  $\nu \in M({\mathbf R})$ by
$$
\nu =
2^{- \frac{r}{p}}
\mu_{m^\prime}
= \sum^{2^{r}}_{j=1} b_j \delta_{N(j)}.
$$
Then we have
$\sum^{2^{r}}_{j=1} |b_j|^p =1$, $\| \widehat{\nu} \|_{L^\infty} < \varepsilon$
and the sets
$\{ - N(j) +  K ~|~  N(j) \in {\rm supp~} \nu \}  $ are mutually disjoint.

\end{remark}

\subsection{The proof of  Theorem \ref{inclusion Mp1 subset Ap}}
$(i)$ We first consider the case $1 \leq p < 2$.
As mentioned at the beginning of Section \ref{A relation between Mp10},
$M^{p,1}({\mathbf R}) \subset {\mathcal F}\hspace{-0.08cm}A_p({\mathbf R})$ holds.
Let $\phi \in C^\infty_c({\mathbf R}) \setminus \{ 0 \}$ be  such that
$K= {\rm supp~} \phi \subset  [ - \frac{1}{10}, \frac{1}{10} ] $.
Then Lemma \ref{existance of discrete measure} yields
that for any $\varepsilon >0$, there exists $\mu \in M({\mathbf R})$
and $m \in {\mathbf N}$
such that $\| \mu \|_{M({\mathbf R})} =1$,
$\| \widehat{\mu} \|_{L^\infty} < \varepsilon$
and
$$
{\rm supp~} \mu =  \{ \ell(j)  \in {\mathbf N}  ~|~ j=1,2, \cdots,    2^m \},
$$
where  the sets
$\{ - \ell(j) +  K ~|~  \ell(j) \in {\rm supp~} \mu \}  $ are mutually disjoint.
It follows from  Remark \ref{express of mu} that
 $\mu$ can be  represented as
$
\mu =
\sum^{2^m}_{j=1}  a_j
 \delta_{\ell (j)}
$
with $\ell(j) \in {\rm supp~} \mu$ and
$a_j =  2^{-m}$ or $a_j = - 2^{-m}$.
Moreover, since $\widehat{\phi} \in L^p({\mathbf R})$,
we see that for any $\eta>0$, there exists a compact subset
 $K_\eta = [-R,R] \subset {\mathbf R}$ for some $R>0$
such that
$
\int_{ {\mathbf R}  \setminus K_\eta } |\widehat{\phi} (\xi) |^p d \xi < \eta.
$
Furthermore,  we have by Remark \ref{modification discrete measure}
that there exist
 $r, N \in {\mathbf N}$ and
$\nu =
\sum^{2^r}_{j=1}  b_j
 \delta_{N(j)}
 \in M({\mathbf R})$
such that
$ \sum^{2^r}_{j=1} |b_j|^p  =1$,
$\| \widehat{\nu} \|_{L^\infty} < \varepsilon$
and
$$
{\rm supp~} \nu =  \{ N(j)  \in {\mathbf N}  ~|~ j=1,2, \cdots,    2^r \},
$$
where  the sets
$\{ - N(j) +  K_\eta ~|~  N(j) \in {\rm supp~} \nu \}  $ are mutually disjoint,
and
$b_j =  2^{ - \frac{r}{p} }  $ or $b_j = - 2^{ - \frac{r}{p}  }$.

Now we define $f$ by
$\widehat{f} = \mu * (\widehat{\nu} \phi)$.\\

\noindent
{\bf Step 1.}
Firstly, we  estimate
$
\| f \|_{M^{p,1}}.
$
Let $\varphi \in {\mathcal S}({\mathbf  R})$ be such that
$\sum_{k \in {\mathbf Z}}\varphi(\xi -k) =1 $,
${\rm supp~} \varphi \subset [-1,1] $ and
$\varphi (\xi) =1$ on $[ - \frac{1}{10}, \frac{1}{10} ]$.
Since
\begin{align*}
\widehat{f} (\xi) =
\Big(\Big(  \sum^{2^m}_{j=1} a_j \delta_{\ell(j)}      \Big) * (\widehat{\nu} \phi )
\Big)(\xi)
=
\sum^{2^m}_{j=1} a_j
(\widehat{\nu}  \phi)
(\xi - \ell(j)),
\end{align*}
we have
$$
\varphi(\xi - \ell(k))
\widehat{f} (\xi)
=  a_k  (   \widehat{\nu}  \phi ) (\xi - \ell(k))
$$
and
\begin{align*}
\varphi(D-k) f (x)
&=
\frac{a_k}{2 \pi}
\int_{{\mathbf R}}
( \widehat{\nu} \phi ) (\xi - \ell(k)) e^{ix \xi} d \xi
=
a_k e^{ix \ell (k)}  {\mathcal F}^{-1} (\widehat{\nu}  \phi) (x).
\end{align*}
Therefore, we obtain
$$
\| \varphi (D-k) f \|_{L^p} = |a_k| \|  {\mathcal F}^{-1} ( \widehat{\nu} \phi ) \|_{L^p}
= |a_k| \| \nu *  ( {\mathcal F}^{-1}\phi ) \|_{L^p}.
$$
On the other hand,
since  $\nu= \sum^{2^r}_{j=1} b_j \delta_{N(j)} $,
we have
\begin{align*}
\| \nu * ( {\mathcal F}^{-1} \phi) \|_{L^p}
= \Big\| \sum^{2^r}_{j=1} b_j ({\mathcal F}^{-1}  \phi) (\cdot  - N(j)) \Big\|_{L^p}.
\end{align*}
Moreover, since $\{ -N(j) + K_\eta \}$ are mutually disjoint
and
$$
\Big| A + \sum^{2^r}_{j=1} B_j \Big|^p
\geq
\frac{1}{2^p} |A|^p - \Big| \sum^{2^r}_{j=1} B_j \Big|^p
\geq
\frac{1}{2^p} |A|^p - 2^{rp}
\sum^{2^r}_{j=1} |B_j|^p,
$$
we obtain that
\begin{align*}
&\Big| \sum^{2^r}_{j=1} b_j  ( {\mathcal F}^{-1}  \phi)  (x  - N(j)) \Big|^p \\
&=
\Big| \sum^{2^r}_{j=1} b_j
\Big( ( \chi_{K_\eta } \cdot  ( {\mathcal F}^{-1}  \phi)) (x  - N(j))
+  ( \chi_{ {\mathbf R}  \setminus K_\eta }  \cdot ( {\mathcal F}^{-1}  \phi)) (x  - N(j))  \Big|^p \\
&
\geq
\frac{1}{2^p}
\sum^{2^r}_{j=1} |b_j |^p
| ( \chi_{K_\eta } \cdot ( {\mathcal F}^{-1}  \phi)) (x  - N(j)) |^p
\\
& \quad - 2^{rp} \sum^{2^r}_{j=1} |b_j|^p |(\chi_{ {\mathbf R}  \setminus K_\eta } \cdot  ( {\mathcal F}^{-1}  \phi)) (x  - N(j)) |^p.
\end{align*}
Since  $\sum^{2^r}_{j=1} |b_j|^p =1 $, we have
\begin{align*}
&\| \nu * ( {\mathcal F}^{-1}  \phi) \|_{L^p}^p \\
&\geq
\sum^{2^r}_{j=1}
|b_j|^p
\Big( \frac{1}{2^p}
\int_{\mathbf R}
| ( \chi_{K_\eta } \cdot  ( {\mathcal F}^{-1}  \phi)) (x  - N(j)) |^p dx \\
&\quad
- 2^{rp} \int_{\mathbf R}|(\chi_{ {\mathbf R}  \setminus K_\eta } \cdot  ( {\mathcal F}^{-1}  \phi)) (x  - N(j)) |^p
dx
\Big) \\
&=
\sum^{2^r}_{j=1}
|b_j|^p
\Big(  \frac{1}{2^p}  \int_{K_\eta} | {\mathcal F}^{-1}  \phi (x)|^p dx - 2^{rp}
\int_{{\mathbf R} \setminus K_\eta} | {\mathcal F}^{-1}  \phi (x)|^p dx \Big) \\
&=
\frac{1}{2^p} \|  {\mathcal F}^{-1}  \phi \|_{L^p}^p -
(1+2^{rp} ) \eta.
\end{align*}
Putting $\eta = \frac{1}{2^{p+1}  ( 1+ 2^{rp} ) } \|  {\mathcal F}^{-1}  \phi \|_{L^p}^p
$,
we obtain
$$
\| \nu * ( {\mathcal F}^{-1}  \phi) \|_{L^p} \geq \frac{1}{2^{1+ \frac{1}{p}}} \|  {\mathcal F}^{-1}  \phi \|_{L^p}.
$$
Hence we have by $\| \mu \|_{M({\mathbf R})} = \sum^{2^m}_{k=1} |a_k|=1$
that
\begin{align*}
\| f \|_{M^{p,1}} =
\sum_{k \in {\mathbf Z}} \| \phi(D-k) f \|_{L^p}
\geq
\frac{1}{2^{1+ \frac{1}{p}}}
\Big(
\sum_{k=1}^{2^m}
|a_k|
\Big) \|  {\mathcal F}^{-1}  \phi \|_{L^1}
= \frac{1}{2^{1+ \frac{1}{p}}} \| {\mathcal F}^{-1}  \phi \|_{L^1}.
\end{align*}

\noindent
{\bf Step 2.}
Secondly, we  show that  $\| \widehat{f}  \|_{L^1} \leq \varepsilon \| \phi \|_{L^1}$.
Since $\widehat{f} = \mu * (\widehat{\nu} \phi)$,
$\| \widehat{\nu} \|_{L^\infty} < \varepsilon$
and $
\mu =
\sum^{2^m}_{j=1}  a_j
 \delta_{\ell (j)}$ with $\| \mu \|_{M({\mathbf R})}
= \sum^{2^m}_{j=1} |a_j| =1$, we obtain
\begin{align*}
\| \widehat{f} \|_{L^1}
& \leq
\sum^{2^m}_{j=1}
|a_j| \| (\widehat{\nu} \phi) (\cdot - \ell(j)) \|_{L^1}
\leq \| \widehat{\nu} \|_{L^\infty} \| \phi \|_{L^1}
\sum^{2^m}_{j=1} |a_j|
< \varepsilon \| \phi \|_{L^1}.
\end{align*}

\noindent
{\bf Step 3.}
Thirdly, we show that  $\| f \|_{L^p} \leq \varepsilon \|  {\mathcal F}^{-1}  \phi \|_{L^p}$.
We note that
\begin{align*}
f (x)
&=  {\mathcal F}^{-1} ( \mu * (\widehat{\nu } \phi)   ) (x) \\
&=
{\mathcal F}^{-1}\Big(
\sum^{2^m}_{j=1} a_j (\widehat{\nu} \phi ) ( \cdot - \ell (j) )
\Big) (x) \\
&=
\frac{1}{2 \pi}
\sum^{2^m}_{j=1}
a_j
\Big(
\int_{\mathbf R }
\widehat{\nu} (\xi - \ell(j)) \phi (\xi - \ell(j))
e^{ix \xi} d \xi
\Big) \\
&=
\sum^{2^m}_{j=1} a_j e^{ix \ell(j)}  {\mathcal F}^{-1} ( \widehat{\nu} \phi ) (x) \\
&=  (\widehat{\nu}(-x) ) ( (\mu * ( {\mathcal F}^{-1}  \phi)) (x)).
\end{align*}
Therefore, we obtain
\begin{align*}
\| f \|_{L^p}
&  \leq \| \widehat{\nu} \|_{L^\infty} \| \mu * ( {\mathcal F}^{-1}  \phi) \|_{L^p}
\leq \| \widehat{\nu} \|_{L^\infty} \| \nu \|_{M({\mathbf R})} \|  {\mathcal F}^{-1}  \phi \|_{L^p}
\leq \varepsilon \|  {\mathcal F}^{-1}  \phi \|_{L^p}.
\end{align*}

\noindent
{\bf Step 4.} We note that if  $M^{p,1} ({\mathbf R}) = {\mathcal F}\hspace{-0.08cm}A_p({\mathbf R})$
set-theoretically, then  the closed graph theorem implies that  the norms
on both spaces are equivalent.
Therefore we have
\begin{align*}
\frac{1}{2^{1+ \frac{1}{p}}} \|  {\mathcal F}^{-1}  \phi \|_{L^1}
\leq \| f \|_{M^{p,1}}
\lesssim
\| f \|_{{\mathcal F}\hspace{-0.08cm}A_p({\mathbf R})}
= \| f \|_{L^p} + \| \widehat{f} \|_{L^1}
\leq \varepsilon (\| \phi \|_{L^p} + \|  {\mathcal F}^{-1}  \phi \|_{L^1}).
\end{align*}
Since $\phi \not= 0$, this is impossible.
Hence we obtain the desired result.\\

\noindent
$(ii)$ Finally, we consider the case $p=2$.
Set
$
I_k =
[
k - \frac{1}{ (\log_e k)^2  }, \
k + \frac{1}{ (\log_e k)^2 }
]
$
$(k \in {\mathbf N})$
and define $f$ by
$
\widehat{f} (\xi)
= \sum^\infty_{k=e^{10 }}
\frac{\chi_{I_k} (\xi) }{k},
$
where $\chi_E$ denotes the characteristic function of a set $E \subset {\mathbf R}$,
i.e., $\chi_E (x)=1$ (if $x \in E$), $=0$ (if $x \not\in E$).
Moreover, let
$\varphi \in {\mathcal S} ({\mathbf R})$ be such that
$\sum_{k \in {\mathbf Z}} \varphi(\xi -k) =1 $ $(\xi \in {\mathbf R})$
and $\varphi (\xi)=1$ on $[-\frac{1}{10}, \frac{1}{10}]$.
Then we have
\begin{align*}
\| \widehat{f} \|_{L^1}
&
\leq
\sum_{k \in {\mathbf Z}}
\| \varphi(\cdot - k) \widehat{f} \|_{L^1}
\leq
\sum_{k= e^{10}}^\infty
\frac{2}{k (\log_e k)^2} < \infty
\end{align*}
and
\begin{align*}
\| f \|_{L^2} \approx
\| \widehat{f} \|_{L^2}
=
\Big(
\sum^\infty_{k=e^{10}}
\frac{ |I_k|  }{k^2}
\Big)^{\frac{1}2}
=
\Big(
\sum^\infty_{k=e^{10}}
\frac{2 }{k^2  (\log_e k)^2  }
\Big)^{\frac{1}{2}}
< \infty.
\end{align*}
Therefore we obtain that  $f \in {\mathcal F}\hspace{-0.08cm}A_2({\mathbf R})$.
On the other hand,
it follows form the Plancherel theorem that
\begin{align*}
\| f \|_{M^{2,1}}
\approx \sum_{k \in {\mathbf Z}} \| \varphi (\cdot - k)  \widehat{f}  \|_{L^2}
=
 \sum_{k =e^{10} }^\infty
\frac{\| \chi_{I_k} \|_{L^2}}{k}
= \sum^\infty_{k= e^{10}}
\frac{2}{ k( \log_e k )  }
= \infty.
\end{align*}
Thus we have $f \not\in M^{2,1} ({\mathbf R})$.
Hence we obtain $f \in {\mathcal F}\hspace{-0.08cm}A_2 ({\mathbf R}) \setminus M^{2,1} ({\mathbf R})$.


\section*{Acknowledgment}
This work was supported by JSPS KAKENHI Grant Numbers 22K03328, 22K03331.

\end{document}